\documentclass[11pt]{article}
\usepackage{latexsym}
\usepackage{amssymb}
\usepackage{amsmath}
\usepackage[all]{xy}
\usepackage{amsthm}
\usepackage{enumerate}
\usepackage{color}
\usepackage{tikz}
\usepackage{graphicx}
\usepackage[misc]{ifsym}

\newtheorem{theorem}{Theorem}[section]
\newtheorem{proposition}[theorem]{Proposition}
\newtheorem{lemma}[theorem] {Lemma}
\newtheorem{corollary}[theorem]{Corollary}

\newtheorem{example}[theorem]{Example}
\newtheorem{remark}[theorem]{Remark}
\newtheorem{definition}[theorem]{Definition}

\def\cprime{$'$}

\def\Im{\mathop{\mathrm{Im}}\nolimits}

\def\Hom{\mathop{\mathrm{Hom}}\nolimits}
\def\Ext{\mathop{\mathrm{Ext}}\nolimits}
\def\Ker{\mathop{\mathrm{Ker}}\nolimits}
\def\dim{\mathop{\mathrm{dim}_k}\nolimits}
\def\HH {\mathop{\mathrm{HH}}\nolimits}
\def \car {\mathop{\mathrm{char}} \nolimits}

\def\id{\mathop{\mathrm{id}}\nolimits}

\def \Z{{\mathbb Z}}

\def\R{\mathcal R}
\def\P{\mathcal P}

\def\C{\mathcal C}
\def\I{\mathcal I}

\def\G{\mathcal G}
\def\E{\mathcal E}

\def\nG{N \mathcal G}
\def\nE{N \mathcal E}

\begin{document}

\title{Gerstenhaber algebra structure on the Hochschild cohomology of quadratic string algebras}
\date{}
\author{Mar\'\i a Julia Redondo  and Lucrecia Rom\'an \footnote{Instituto de Matem\'atica (INMABB), 
Departamento de Matem\'atica, Universidad Nacional del Sur (UNS)-CONICET, Bah\'\i a Blanca, Argentina.
{\it E-mail address: mredondo@criba.edu.ar, lroman@uns.edu.ar}}
\thanks{The first author is a researcher and the second author has a fellowship from CONICET, Argentina. This work has been supported by the project PICT-2011-1510.  }}

\maketitle

\begin{abstract}
We describe the Gerstenhaber algebra structure on the Hochschild cohomology  $\HH^*(A)$ when $A$ is a quadratic string algebra. First we compute the Hochschild cohomology groups using Barzdell's resolution and we describe generators of these groups. Then we construct comparison morphisms between the bar resolution and Bardzell's resolution in order to get formulae for the cup product and the Lie bracket.  We find conditions on the bound quiver associated to  string algebras in order to get non-trivial structures.
\end{abstract}


\section{Introduction}

Let $A$ be an associative, finite dimensional algebra over an algebraically closed field $k$. By a fundamental result in representation theory  it is well known that there exists a finite quiver $Q$ such that $A$ is Morita equivalent to $kQ/I$, where $kQ$ is the path algebra of $Q$ and $I$ is an admissible two-sided ideal of $kQ$,  see for instance \cite[Theorem 3.7]{ASS}. The pair $(Q,I)$ is called a bound quiver of $A$.

A finite dimensional algebra is called {\bf biserial} if the radical of every projective indecomposable module is the sum of two uniserial modules whose intersection is simple or zero, see~\cite{F}.
These algebras have been studied by several authors and from different points of view since there are a lot of natural examples of algebras which turn out to be of this kind:  Nakayama algebras~\cite{N}, Kawada algebras~\cite{K,R1}, blocks of group algebras with cyclic or dihedral defect groups~\cite{J,E,R0}, iterated tilted algebras  of type $A$ and $\tilde{A}$~\cite{AH, AS}, Brauer graph algebras~\cite{SS} and algebras appearing in the Gel{\cprime}fand-Ponomarev classification of Harish-Chandra modules over the Lorentz group~\cite{GP}.

The representation theory of these algebras was first studied by Gel{\cprime}fand and Ponomarev in~\cite{GP}: they have provided the methods in order to classify all their indecomposable representations. This classification shows that biserial algebras are always tame, see also~\cite{CB}. They are an important class of algebras whose representation theory has been very well described, see~\cite{AS,BR}.

The subclass of {\bf special biserial algebras} was studied by Skowro\'nski and Waschb{\"u}sch in~\cite{SW} where they  characterize
the biserial algebras of finite representation type. The definition of these algebras can be given in terms of conditions on the bound quiver $(Q,I)$ associated (see Section \ref{definiciones}). A classification of the special biserial algebras which are minimal representation-infinite has been given by Ringel in~\cite{R}.   This class of algebras has played an important r\^{o}le in the study of self-injective algebras, and their representation theory is well understood. There is a beautiful description of all  finite-dimensional indecomposable modules
over special  biserial  algebras:  they  are either  string modules  or
band modules or non-uniserial projective-injective modules,  see~\cite{BR, WW}.

Since $A$ is an algebra over a field $k$, the {\bf Hochschild cohomology groups} ${\HH}^n(A,M)$ with coefficients in an $A$-bimodule $M$ can be identified with the groups $\Ext_{A-A}^n(A,M)$.  In particular, if $M$ is the $A$-bimodule $A$, we simply write ${\HH}^n(A)$.  Even though the computation of the Hochschild cohomology groups ${\HH}^n(A)$ is rather complicated, some approaches have been successful when the algebra $A$ is
given by a quiver with relations. For instance, explicit formula for the dimensions of ${\HH}^n(A)$ in terms of
those combinatorial data have been found in~\cite{blm,c1,c2,crs,H, Re}. In particular, Hochschild cohomology of special biserial algebras has been considered in~\cite{Bu,L,ST}.

In the particular case of monomial algebras, that is, algebras $A=kQ/I$ where $I$ can be chosen as generated by paths, one has a detailed description of a minimal resolution of the $A$-bimodule $A$:  Bardzell's resolution~\cite{B}. 

An algebra is called a  {\bf string algebra} if it is Morita equivalent to a monomial special biserial algebra $kQ/I$, and it is called {\bf quadratic} if the ideal $I$ is generated by paths of length two. In general, the computation of the Hochschild cohomology groups using Bardzell's resolution may lead to hard combinatorial computations.  However, for quadratic string algebras the resolution, and the complex associated, are easier to handle. 

The sum ${\HH}^*(A) = \bigoplus_{n \geq 0} {\HH}^n(A)$ is a Gerstenhaber algebra, that is, it is a graded commutative ring via the cup product,  a graded Lie algebra via the bracket, and these two structures are related, see~\cite{G1}.
So far there are only a few classes of algebras in the literature where the Gerstenhaber algebra structure on Hochschild cohomology has been determined explicitly, see \cite{Bu, LZ, LZ2, SF, ST, SC, SWa, SA, WZ}. In~\cite{RR} we  compute the Hochschild cohomology groups of a triangular string algebra, and we prove that its ring structure is trivial.

The purpose of this paper is to study the Gerstenhaber algebra structure on the Hochschild cohomology of a quadratic string algebra $A$, that is,  we describe explicitly the  ring and the Lie algebra structure of the Hochschild cohomology of $A$, and we delete the triangular condition in order to get non-trivial structures.
The main results show the dimension of the Hochschild cohomology groups $\HH^n(A)$, see Theorems \ref{0-1} and \ref{n>1}.  These computations are done using Bardzell's resolution, and allow us to give  an  explicit  basis  for  each cohomology  group.  Furthermore,  we construct a comparison morphism between the bar resolution and Bardzell's resolution which lead us to the formulae for the cup product and the Lie bracket.

Finally we describe explicit non-zero elements in the Hochschild cohomology whose cup product and Lie bracket are non-zero, that is, we define $\G_n$ to be the set of {\bf gentle pairs} (see Definition \ref{defpargentil}(a)) and we prove the following theorems.

\medskip

{\bf Theorem} (Theorem \ref{cup}) {\it 
Let $A = kQ/I$ be a quadratic string algebra
and $\G_n \neq \emptyset$ for some  $n > 0$. Then the cup product
defined in $\HH^{*}(A)$ is non-trivial.  More
precisely,
\begin{itemize}
\item [(i)] if $n$ is even and $\car k \not =2$,  $\HH^{s_1n}(A) \cup \HH^{s_2n} (A) \not = 0$;
\item [(ii)] if $n$ is odd and $\car k \not =2$, $\HH^{2s_1n}(A) \cup \HH^{2s_2n} (A) \not = 0$;
\item [(iii)] if $\car k  =2$,  $\HH^{s_1n}(A) \cup \HH^{s_2n} (A) \not = 0$
\end{itemize}
for any $s_1, s_2 \geq 1$.}

\bigskip

{\bf Theorem} (Theorem \ref{lie}) {\it 
Let $\car k =0$ and let $A = kQ/I$ be a quadratic string algebra such that $\G_n
\neq \emptyset$ for some  $n > 0$. Then the Lie bracket defined
in $\HH^{*}(A)$ is non-trivial. More
precisely,
\begin{itemize}
\item [(i)] if $n$ is even,  $[\HH^{s_1n+1}(A) , \HH^{s_2n+1} (A)] \not = 0$;
\item [(ii)] if $n$ is odd,  $[\HH^{2s_1n+1}(A) , \HH^{2s_2n+1} (A)] \not = 0$
\end{itemize}
for any $s_1, s_2 \geq 1$, $s_1 \not = s_2$.}

\bigskip

These proofs are given in terms of representatives $u_{sn}$ and $u_{sn+1}$ of non-zero elements in $\HH^{sn}(A)$ and $\HH^{sn+1}(A)$ respectively such that 
$$u_{s_1n} \cup u_{s_2n} = u_{(s_1 + s_2)n} \quad \mbox{and} \quad [ u_{s_1n+1} , u_{s_2n+1}] = \lambda u_{(s_1+s_2)n+1}$$ for some scalar $\lambda \in k$.

The paper is organized as follows.  In Section 2 we introduce all the necessary terminology.  In Section 3 we recall the resolution given by Bardzell for monomial algebras in~\cite{B} and we present all the computations that lead us to Theorems \ref{0-1} and \ref{n>1} where we present the dimension of all the Hochschild cohomology groups of quadratic string algebras. In Section 4 we describe the ring and the Lie structure of the Hochschild cohomology of these algebras and we find conditions on the bound quiver associated in order to get non-trivial structures.

\section {Preliminaries}

\subsection{Quivers and relations}

Let $Q$ be a connected finite quiver with a set of vertices $Q_0$, a set of arrows $Q_1$ and $s, t : Q_1 \to Q_0$ be the maps associating to each arrow $\alpha$ its source  $s(\alpha)$ and its target $t(\alpha)$.  A path $w$ of {\bf length} $l$ is a sequence of $l$ arrows $\alpha_1 \dots \alpha_l$ such that $t(\alpha_i)=s(\alpha_{i+1})$. We denote by $\vert w \vert $ the length of the path $w$. We put $s(w)=s(\alpha_1)$ and $t(w)=t(\alpha_l)$. For any vertex $x$ we consider $e_x$ the trivial path of length zero and we put $s(e_x)=t(e_x)=x$.  An oriented cycle is a non-trivial path $w$ such that $s(w)=t(w)$.

The {\bf path algebra} $kQ$ is the $k$-vector space with basis the set of paths in $Q$; the product on the basis elements is given by the concatenation of the sequences of arrows of the paths $w$ and $w'$ if they form a path (namely, if $t(w)=s(w')$) and zero otherwise.  Vertices form a complete set of orthogonal idempotents of $kQ$. Let $F$ be the two-sided ideal of $kQ$ generated by the arrows of $Q$. A two-sided ideal $I$  of $kQ$ is said to be {\bf admissible} if there exists an integer $m \geq 2$ such that $F^m \subseteq I \subseteq F^2$.  The elements in $I$ are called {\bf relations}, $kQ/I$ is called a {\bf monomial algebra} if the ideal $I$ is generated by paths, and a relation is called {\bf quadratic} if it is a path of length two.

By a fundamental result in representation theory  it  is well known that if $A$ is a basic, indecomposable, finite dimensional algebra over an algebraically closed field $k$, then there exists a unique finite  connected quiver $Q$ and a surjective morphism of $k$-algebras $\nu: kQ \to A$, which is not unique in general, with $I_\nu=\Ker \nu$ admissible,    see for instance \cite[Theorem 3.7]{ASS}.  The pair $(Q,I_\nu)$ is called a {\bf presentation} of $A$.  

\subsection{String algebras}\label{definiciones}

Recall from~\cite{SW} that a bound quiver $(Q,I)$ is  {\bf special biserial} if it satisfies the following conditions:
\begin{itemize}
\item [S1)] Each vertex in $Q$ is the source of at most two arrows and the target of at most two arrows;
\item [S2)] For an arrow $\alpha$ in $Q$ there is at most one arrow $\beta$ and at most one arrow $\gamma$ such that $\alpha\beta \not \in I$ and $ \gamma \alpha \not \in I$.
\end{itemize}
If the ideal $I$ is generated by paths, the bound quiver $(Q,I)$ is  {\bf string}.

An algebra is called {\bf special biserial} (or {\bf string}) if it
is Morita equivalent to a path algebra $kQ/I$ with $(Q,I)$ a special
biserial bound quiver (or a string bound quiver, respectively).

Since Hochschild cohomology is invariant under Morita equivalence,
whenever we deal with a string algebra $A$ we will assume that it is
given by a string presentation $A = kQ/I$  with $(Q,I)$ satisfying the
previous conditions.

In this paper we are interested in quadratic string algebras, that is, the ideal $I$ is generated by paths of length two.  We fix a minimal set $\R$ of paths that generates the ideal $I$.  Moreover, we denote by $\P$ the set of paths in 
$Q$ such that the set $\{ \gamma + I, \gamma \in \P\} $ is a basis of $A=kQ/I$. It is clear that $Q_0 \cup Q_1 \subseteq \P$ since $I \subseteq F^2$.

\section{Hochschild cohomology groups}
\subsection{Bardzell's resolution for quadratic string algebras} 

We recall that the Hochschild cohomology groups ${\HH}^n(A)$ of an algebra $A$ are the groups $\Ext_{A-A}^n(A,A)$, $n \geq 0$.  Since string algebras are monomial algebras, their Hochschild cohomology groups can be computed using a convenient minimal projective resolution of $A$ as $A$-bimodule constructed by Bardzell in~\cite[Theorem 4.1]{B}. In the particular case of quadratic string algebras, this minimal resolution is the following:
$$ \dots \longrightarrow A \otimes k AP_n \otimes A \stackrel{d_{n}}{\longrightarrow} A \otimes k AP_{n-1} \otimes A \longrightarrow \dots \longrightarrow  A \otimes k AP_0 \otimes A \stackrel{\mu}{\longrightarrow} A \longrightarrow 0$$
where $kAP_0=kQ_0$, $kAP_1=kQ_1$, for $n \geq 2$, $kAP_{n}$ is the vector space generated by the set
$$AP_n = \{ \alpha_1 \alpha_2 \cdots \alpha_n :  \alpha_i \alpha_{i+1} \in I, 1 \leq i < n\},$$
all tensor products are taken over $E=kQ_0$, the subalgebra of $A$ generated by the vertices, and the $A$-bimodule morphisms are
\begin{align*}
\mu( 1 \otimes e_i \otimes 1 ) & =  e_i, \\
d_1( 1 \otimes \alpha \otimes 1) & =   \alpha \otimes e_{t(\alpha)} \otimes 1 - 1 \otimes e_{s(\alpha)} \otimes \alpha, \\
d_{n} (1 \otimes \alpha_1 \cdots \alpha_n \otimes 1) & = \alpha_1 \otimes \alpha_2 \cdots \alpha_n \otimes 1 + (-1)^n 1 \otimes \alpha_1 \cdots \alpha_{n-1} \otimes \alpha_n.
\end{align*}
The $E$-$A$-bilinear map $c: A \otimes k AP_{n-1} \otimes A \to A \otimes k AP_{n} \otimes A$ defined by
\[ c( a \otimes \alpha_1 \cdots \alpha_{n-1} \otimes 1) = \begin{cases}
b \otimes \alpha_0 \alpha_1 \cdots \alpha_{n-1}  \otimes 1 & \mbox{if $a=b\alpha_0$ and $\alpha_0 \alpha_1 \in I$,} \\
0 & \mbox{otherwise}
\end{cases}\]
is a contracting homotopy, see~\cite[Theorem 1]{S} for more details. \\

The Hochschild complex, obtained by applying $\Hom_{A-A} (-, A)$ to the resolution we have just described and using the isomorphisms
\[ \Hom_{A-A}( A \otimes k AP_n \otimes A,  A)  \simeq \Hom_{E-E}( k AP_n,  A)\]
is
\[ 0 \longrightarrow \Hom_{E-E} (k AP_0 , A)  \stackrel{F_{1}}{\longrightarrow} \Hom_{E-E}( k AP_1,  A ) \stackrel{F_{2}}{\longrightarrow}   \Hom_{E-E}(  k AP_2, A)  \to  \cdots \]
where
\begin{align*}
F_1( f) (\alpha)  & =   \alpha f(e_{t(\alpha)}) - f(e_{s(\alpha)}) \alpha,\\
F_{n}(f) (\alpha_1 \cdots \alpha_n)  & =  \alpha_1 f(\alpha_2 \cdots \alpha_n) + (-1)^n f(\alpha_1 \cdots \alpha_{n-1}) \alpha_n.
\end{align*}

\subsection{Computations}

In order to describe the ring and the Lie algebra structure of the Hochschild cohomology ${\HH}^*(A)$ of a quadratic string algebra, we need a convenient description of the previous complex.  In this section we will describe explicit basis of these $k$-vector spaces and study the behavior of the maps between them in order to get information about kernels and images.

For any pair $X,Y$ of sets of paths in $Q$ we denote by $(X// Y)$ the set of
pairs $(\rho, \gamma) \in X \times Y$ such that $\rho, \gamma$ are
parallel paths in $Q$, that is
\[ (X// Y) = \{  (\rho, \gamma) \in X \times Y: s(\rho) = s(\gamma),   t(\rho) =
t(\gamma)\}.\]
Recall that we have fixed a set $\P$ of paths in $Q$ such that the set $\{ \gamma + I: \gamma \in \P\} $ is a basis of $A=kQ/I$.  For any $m \geq 0$ we denote by $\P_m$ the set of paths in $\P$ whose length is greater than or equal to $m$. Then
$$(X// \P_m) = \{  (\rho, \gamma) \in  X \times \P  : s(\rho) = s(\gamma),   t(\rho) =
t(\gamma), \vert \gamma \vert \geq m\}$$
and hence  $(X// \P)=(X// Q_0) \sqcup (X// \P_1) = (X// Q_0) \sqcup  (X// Q_1) \sqcup (X// \P_2) $, where $\sqcup$ depicts disjoint union.
\medskip

\begin{example}\label{ejemplo1}
 Let  $A=kQ/I$  where $Q$ is the quiver 
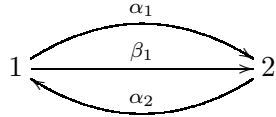
\[ \xymatrix{
 1 \ar@/^{6mm}/[rrr]^{\alpha_1} \ar[rrr]^{\beta_1}  & & & 2\ar@/^{6mm}/[lll]_{\alpha_2}   
}\] and    $I = < \alpha_1\alpha_2, \alpha_2\alpha_1,
 \beta_1\alpha_2>$. The sets $(AP_n//\P)$ are 
 \begin{align*} (AP_0//Q_0) & =   \{ (e_1,e_1), (e_2, e_2)\}, \\
(AP_0//\P_1) & =    \{ (e_2,\alpha_2\beta_1) \} 
\end{align*}
and,  for  $i \geq 0$, 
\begin{align*}
(AP_{2i}//Q_0) & =   \{ ((\alpha_1\alpha_2)^{i}, e_1), ( (\alpha_2\alpha_1)^{i},  e_2),  (\beta_1\alpha_2(\alpha_1\alpha_2)^{i-1}, e_1) \}, \\
(AP_{2i}//\P_1) & =   \{   ((\alpha_2\alpha_1)^{i}, \alpha_2\beta_1) \}, \\
(AP_{2i+1}//Q_0) & =   \emptyset, \\
(AP_{2i+1}//\P_1) & =   \{ ((\alpha_1\alpha_2)^{i} \alpha_1,  \alpha_1),  ((\alpha_1\alpha_2)^{i} \alpha_1, \beta_1), 
 ((\alpha_2\alpha_1)^{i}\alpha_2, \alpha_2), \\
 & \qquad   (\beta_1(\alpha_2\alpha_1)^{i}, \alpha_1), (\beta_1(\alpha_2\alpha_1)^{i}, \beta_1)\}.
\end{align*}  
 \end{example}

\begin{example}\label{ejemplo2} Let  $A = kQ/I$  where $Q$ is the quiver
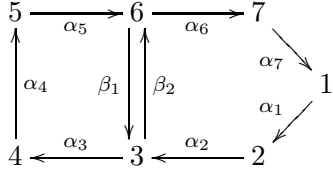
\[\xymatrix@=5mm{ 5 \ar[rr]_{\alpha_5} & &  6
\ar@<-3pt>[dd]_{\beta_1} \ar[rr]_{\alpha_6} & & 7
\ar[rd]_{\alpha_{7}} \\  & & & & &  1 \ar[ld]_{\alpha_{1}}
\\ 4 \ar[uu]_{\alpha_{4}} & & 3 \ar[ll]_{\alpha_{3}}
\ar@<-3pt>[uu]_{\beta_{2}} & & 2 \ar[ll]_{\alpha_{2}}}
\]
and   $I = < \alpha_i\alpha_{i+1}>_{ \{i=1,\cdots,6 \} }  + <
\beta_{1}\beta_2, \  \beta_{2}\beta_1  > $.  The sets $(AP_n//\P)$ are 
\begin{align*} 
(AP_0//Q_0) & =  \{ (e_i,e_i)\}_{ \{i=1,\cdots,7 \} }, \\
(AP_0//\P_1) & =    \emptyset, \\   
(AP_1//Q_0) & =    \emptyset, \\
(AP_1//\P_1) & =    \{ (\alpha_i, \alpha_i)\}_{ \{i=1,\cdots,7 \}} \sqcup \{(\beta_i, \beta_i)\}_{ \{i=1,2 \}},\\ 
(AP_{2i}//Q_0) & =   \{ ((\beta_1\beta_2)^i, e_6), ((\beta_2\beta_1)^i, e_3)  \} \quad \mbox{for  $i \geq 1$}, \\
(AP_{4}//\P_1) & =  \{ (\alpha_2\alpha_3\alpha_4\alpha_5, \alpha_2\beta_2), (\alpha_3\alpha_4\alpha_5\alpha_6, \beta_2\alpha_6) \},\\
(AP_{2i}//\P_1) & = \emptyset  \quad  \mbox{for $i \geq 1, i \not = 2$},\\
(AP_{7}//Q_0) & =   \{ (\alpha_1\alpha_2 \dots \alpha_6\alpha_7, e_1) \}, \\
(AP_{2i+1}//Q_0) & =  \emptyset  \quad \mbox{for $i \geq 1, i \not = 3$},\\
(AP_{3}//\P_1) & =  \{ (\alpha_3\alpha_4\alpha_5, \beta_2), (\beta_1\beta_2\beta_1, \beta_1), (\beta_2\beta_1\beta_2, \beta_2)\}, \\
(AP_{5}//\P_1) & =  \{ (\alpha_2\alpha_3\alpha_4\alpha_5\alpha_6, \alpha_2\beta_2\alpha_6), ((\beta_1\beta_2)^2\beta_1, \beta_1), ((\beta_2\beta_1)^2\beta_2, \beta_2)\}, \\
(AP_{2i+1}//\P_1) & =  \{ ((\beta_1\beta_2)^i\beta_1, \beta_1), ((\beta_2\beta_1)^i\beta_2, \beta_2)\}  \quad \mbox{for $i >2$}.
\end{align*}
\end{example}

\medskip

Observe that the  $k$-vector spaces  $\Hom_{E-E} (k AP_{n} , A)$ and $k(AP_{n}// \P)$ are isomorphic: the basis element $(\rho, \gamma) \in (AP_{n}// \P)$ corresponds to the morphism $f_{(\rho, \gamma)}$ in $\Hom_{E-E} (k AP_{n} , A)$ defined by
\[ f_{(\rho, \gamma)} (w)= \begin{cases}
\gamma & \mbox{if $w=\rho$}, \\
0 & \mbox{otherwise}.
\end{cases}\]
Taking into account these isomorphisms, for any $n \geq 0$ we can
construct commutative diagrams
{\small
\[ \xymatrix{
\Hom_{E-E} (k AP_{n} , A) \ar^\cong[dd] \ar^{F_{n+1}}[rrr] & & &  \Hom_{E-E} (k AP_{n+1} , A)  \ar^\cong[dd]  \\ \\
k(AP_{n}// Q_0) \oplus k(AP_{n}// \P_1) \ar@{-->}^{ \left ( \begin{matrix} 0 & 0 \\ F^0_{n+1}  & 0 \\ 0 & F^1_{n+1}
\end{matrix} \right ) \qquad \qquad 
}[rrr] & & &   k(AP_{n+1}// Q_0) \oplus  k(AP_{n+1}// Q_1)
\oplus k(AP_{n+1}// \P_2) }\]}
where
\begin{align*}
F_1^0 (e_r, e_r) & =   \sum_{ \{ \beta \in Q_1:  \ t(\beta)=r \} } (\beta , \beta)  - \sum_{ \{ \beta \in Q_1: \ s(\beta)=r  \} } (\beta, \beta ),     \\
F_1^1 (e_r, \gamma) & =   \sum_{ \{ \beta \in Q_1:  \ t(\beta)=r \} } (\beta , \beta \gamma)  - \sum_{ \{ \beta \in Q_1: \ s(\beta)=r  \} } (\beta, \gamma \beta ),    
\end{align*} 
\begin{align*}
F^0_{n+1} & (\alpha_1 \cdots \alpha_{n}, e_{s(\alpha_1)})  \\
& =  \sum_{\substack{ \{ \beta \in Q_1: \ \beta \alpha_1 \in I, \\ t(\beta)=s(\alpha_1) \} }} (\beta \alpha_1 \cdots \alpha_{n}, \beta)  + (-1)^{n+1} \sum_{ \substack{ \{ \beta \in Q_1: \  \alpha_{n} \beta \in I , \\ s(\beta)=s(\alpha_1)\} }} (\alpha_1 \cdots \alpha_{n} \beta, \beta ),
\\
F^1_{n+1} & (\alpha_1 \cdots \alpha_{n}, \gamma)  \\
&  =  \sum_{\substack{ \{ \beta \in Q_1:  \\ \beta \alpha_1 \in I, \ \beta \gamma \not \in I  \} }} (\beta \alpha_1 \cdots \alpha_{n}, \beta \gamma)  + (-1)^{n+1} \sum_{ \substack{ \{ \beta \in Q_1: \\ \alpha_{n} \beta \in I , \ \gamma \beta \not \in I \} }} (\alpha_1 \cdots \alpha_{n} \beta, \gamma \beta )   .
\end{align*}

Now we will introduce several subsets of $(AP_{n}// \P)$  in order to get a nice description of the kernel and the image of $F^0_{n+1}$ and $F^1_{n+1}$. \\

For $n=0$ we have that $(AP_0//\P) = (Q_0 // \P)$.  Since $A$ is finite dimensional and quadratic, if $(e_{s(\alpha_1)}, \alpha_1 \cdots  \alpha_m) \in (Q_0 // \P_1)$ then $\alpha_m \alpha_1 \in I$.

For $n=1$, $(AP_1// \P) = (Q_1//\P)$, and we consider the following partition
\[(Q_1//\P) = (1,1)_1 \sqcup (0,0)_1 \sqcup (1,0)_1 \sqcup (0,1)_1\]
defined as follows: if $(\alpha, \gamma) \in (Q_1//\P)$, we wonder if the path $\gamma$ starts or ends with the arrow $\alpha$, that is, 
\begin{align*}
(1,1)_1 &= \{ (\alpha, \alpha) : \alpha \in Q_1 \}, \\
(0,0)_1 &= \{ (\alpha, \gamma)  \in   (Q_1//\P) :  \gamma \not \in \alpha kQ \cup kQ \alpha \},\\
(1,0)_1 &= \{ (\alpha, \gamma)  \in   (Q_1//\P) :   \gamma \in \alpha kQ, \gamma \not \in kQ \alpha \},\\
(0,1)_1 &= \{ (\alpha, \gamma)  \in   (Q_1//\P) :  \gamma \not \in \alpha kQ, \gamma \in kQ \alpha  \}.
\end{align*}
For any $n \geq 2$, if $(\rho, \gamma) \in (AP_{n} // \P)$ then $\rho$ is a path of length greater than or equal to $2$, that is, $\rho= \alpha_1 \hat {\rho} \alpha_2$ and wondering if the path $\gamma$ starts with the arrow $\alpha_1$ or ends with the arrow $\alpha_2$, we get the following partition
$$ (AP_{n} // \P) = (0,0)_n \sqcup  (1,0)_n \sqcup  (0,1)_n \sqcup  (1,1)_n$$
where
\begin{align*}
(0,0)_n &= \{ (\rho, \gamma) \in (AP_{n} // \P) : \rho= \alpha_1 \hat {\rho} \alpha_2 \ \mbox{and} \ \gamma \not \in \alpha_1 kQ \cup kQ \alpha_2\}, \\
 (1,0)_n &= \{ (\rho, \gamma) \in (AP_{n} // \P) : \rho= \alpha_1 \hat {\rho} \alpha_2 \ \mbox{and} \ \gamma  \in \alpha_1 kQ, \gamma \not \in  kQ \alpha_2\} ,\\
 (0,1)_n &= \{ (\rho, \gamma) \in (AP_{n} // \P) : \rho= \alpha_1 \hat {\rho} \alpha_2 \ \mbox{and} \ \gamma \not \in \alpha_1 kQ, \gamma \in  kQ \alpha_2\}, \\
(1,1)_n &= \{ (\rho, \gamma) \in (AP_{n} // \P) : \rho= \alpha_1 \hat {\rho} \alpha_2 \ \mbox{and} \ \gamma \in  \alpha_1 kQ \cap kQ\alpha_2 \}.
\end{align*}
We also have to distinguish elements inside each of the previous sets taking into account the following definitions:
\begin{align*}
 ^+(X // \P) & =  \{ (\rho, \gamma) \in (X // \P) : Q_1 \gamma \not \subset I  \}, \\
 ^-(X // \P) & =  \{ (\rho, \gamma) \in (X // \P) : Q_1 \gamma  \subset I \}.
\end{align*}
In an analogous way we define $(X// \P)^+$,  $(X// \P)^-$,   $^+(X// \P)^+=  ^+(X// \P) \cap  (X// \P)^+$ and so on.

\begin{example}   For the algebra presented in   Example \ref{ejemplo1} we  have that, for $n>0$, 
\begin{align*}^+(0,0)^+_{n} & =   \begin{cases} (AP_{n} // Q_0) & \mbox{ if $n$ is even}, \\
\emptyset &  \hbox{ otherwise,}  \end{cases}\\ 
^+(0,0)^-_{n} & =    \begin{cases} \{ ((\alpha_1\alpha_2)^i \alpha_1, \beta_1)\} & \mbox{if $n= 2i+1$,}\\
\emptyset &  \hbox{ otherwise,}  \end{cases} \\
^-(0,0)^-_{n} & =    \begin{cases} \{ (\beta_1, \alpha_1 )\} & \mbox{ if $n= 1$},\\
\emptyset &  \hbox{ otherwise,}  \end{cases}\\
(1,0)_{n}^{-} &=  \begin{cases} 
 \emptyset, & \hbox{if $n= 1$},\\
 \{ (\beta_1(\alpha_2\alpha_1)^i, \beta_1)\} & \mbox{if $n= 2i+1$, $i \geq 1$}, \\
  \{   ((\alpha_2\alpha_1)^{i}, \alpha_2\beta_1) \}  &  \mbox{if $n= 2i$,}  \end{cases} \\
^-(0,1)_{n} & =  \begin{cases} \{  (\beta_1(\alpha_2\alpha_1)^i, \alpha_1) \} & \mbox{ if $n= 2i+1$, $i \geq 1$} ,\\ 
\emptyset &  \hbox{ otherwise}  \end{cases} \\
(1,1)_{n} &= \begin{cases}  \{ (\alpha_1, \alpha_1), (\alpha_2, \alpha_2), (\beta_1, \beta_1) \} & \hbox{if $n=1$},\\
\{((\alpha_1\alpha_2)^i \alpha_1,  \alpha_1),  ((\alpha_2\alpha_1)^i\alpha_2, \alpha_2)\} & \mbox{if $n = 2i+1$, $i \geq 1$}, \\ 
\emptyset &  \hbox{ otherwise,}  \end{cases} \\
 ^-(0,0)^+_{n}  & =  (1,0)_{n}^{+} \  = \ ^+(0,1)_{n} \ = \  \emptyset, 
\end{align*}
and in Example \ref{ejemplo2} we have that
\begin{align*} 
^+(0,0)^+_{n} & =  \begin{cases}  
(AP_{n}//Q_0) &   \mbox{ if $n$ is even}, \\
\{ (\alpha_3\alpha_4\alpha_5, \beta_2) \} &   \mbox{ if $n= 3$}, \\
 (AP_7//Q_0) &  \mbox{ if $n= 7$}, \\
\emptyset &  \hbox{ otherwise,}  \end{cases}\\ 
(1,0)^+_n &=   \begin{cases} \{(\alpha_2\dots\alpha_5, \alpha_2\beta_2)\} & \mbox{if $n= 4$}, \\
\emptyset, &  \hbox{ otherwise,}  \end{cases} \\
^+(0,1)_n &=   \begin{cases} \{(\alpha_3\dots\alpha_6, \beta_2\alpha_6)\} & \mbox{if $n= 4$}, \\
\emptyset, &  \hbox{ otherwise,}  \end{cases} \\
(1,1)_{3} & = \{ (\beta_1\beta_2) \beta_1, \beta_1), (\beta_2\beta_1\beta_2, \beta_2) \},\\
(1,1)_{n} & = (AP_n// \P_1) \qquad   \hbox{ if $n$ is odd, $n \not =3$},\\
(1,1)_{n} & =\emptyset \qquad \qquad \hbox{ if $n$ is even,}\\
^+(0,0)^-_{n} & =  ^-(0,0)^+_{n} \ = \  ^-(0,0)^-_{n} \ = \ (1,0)^-_n \ = \ ^-(0,1)_n  \ = \ \emptyset.
\end{align*} 

\end{example}

\medskip

We start by describing the behavior of our maps $F^1_{n+1}$ restricted to the subsets we have just defined. With this aim in mind, for any $n \geq 1$  we define a map
$$\phi_n: (1,0)_{n}^+ \to {^+(0,1)_{n}}$$ as follows:  given $(\rho, \gamma) \in (1,0)_{n}^+$ we have that $\rho = \alpha \rho'$, $\gamma= \alpha \gamma'$, $\rho'$ and $\gamma'$ do not share their ending arrows and there exists an arrow $\beta$ such that $\gamma' \beta \not \in I$.  Since $A$ is a string algebra we conclude that $\beta$ is unique and that $(\rho' \beta, \gamma' \beta) \in {^+(0,1)_{n}}$.  Hence we define $\phi_n (\alpha \rho', \alpha \gamma') = (\rho' \beta, \gamma' \beta)$.

The proof of the following lemma is analogous to the proof of~\cite[Lemma 4.1]{RR} concerning triangular string algebras. 

\begin{lemma}\label{nucleo}
For any $n \geq 1$ we have
\begin{itemize}
\item[(a)] $^-(0,0)_{n}^- \sqcup (1,0)_{n}^- \sqcup \ ^-(0,1)_{n} \sqcup  (1,1)_{n} \subset \Ker F^1_{n+1}$;
\item[(b)] the function $F^1_{n+1}$ induces a bijection from $^-(0,0)_{n}^+$ to $^{-}(0,1)_{n+1} \cap
( AP_{n+1} // \P_2) $;
\item[(c)] the function $F^1_{n+1}$ induces a bijection from $^+(0,0)_{n}^-$ to $(1,0)^{-}_{n+1} \cap
( AP_{n+1} // \P_2)$;
\item[(d)] the maps $\phi_n: (1,0)_{n}^+ \to {^+(0,1)_{n}}$ are bijections and 
\[ \begin{aligned}
(id +(-1)^{n} \phi_{n}) ((1,0)_{n}^+) & \subset \Ker F^1_{n+1}, \\
(-1)^{n+1} F^1_{n+1} ((1,0)_{n}^+) & = (1,1)_{n+1} \cap ( AP_{n+1} // \P_2), \\
F^1_{n+1} (^+(0,0)^+_{n} \cap ( AP_n // \P_1     ) )  & =  (id  +
(-1)^{n+1} \phi_{n+1}) ((1,0)^+_{n+1})  \cap k( AP_{n+1} // \P_2).
\end{aligned} \]
\end{itemize}
\end{lemma}

\begin{proof}
 \begin{itemize}
\item[(a)] In order to check that $(\rho, \gamma) \in (AP_n // \P_1)$ belongs to $\Ker F^1_{n+1}$ we have to prove that if $ \beta_1 \rho$ or $\rho\beta_2$
belong to $AP_{n+1}$ then $\beta_1\gamma \in I$ and $\gamma \beta_2 \in I$. \\
If $(\rho, \gamma) \in {^-(0,0)_{n}^-}$  then the statement is clear because $Q_1 \gamma \subset I$  and $\gamma Q_1 \subset I$.\\
If $(\rho, \gamma) \in (1,0)_{n}^-$, namely $(\rho, \gamma) =  (\alpha\hat{\rho}, \alpha\hat{\gamma} )$ with $\gamma Q_1 \subset I$,  then $\gamma \beta_2 \in I$.  On the other hand,   $\beta_1 \rho \in AP_{n+1}$ implies that $\beta_1\alpha \in I$ and then  $\beta_1\gamma \in I$.\\ 
The proof for ${^-(0,1)_{n}}$ is analogous. \\
Finally, if $(\rho, \gamma) \in (1,1)_{n}$, this is,  $(\rho, \gamma) = (\alpha, \alpha) $ if $n = 1$ and if $n \geq 2$,  $(\rho, \gamma) = (\alpha_1\hat{\rho}\alpha_2, \alpha_1\hat{\gamma}\alpha_2) $, again $\beta_1 \rho, \rho\beta_2 \in  AP_{n+1}$ imply directly   $\beta_1\alpha, \alpha\beta_2 \in I$ in the first case, and for $n \geq 2$, $\beta_1\alpha_1, \alpha_2\beta_2 \in I$, so $\beta_1\gamma \in I$ and $\gamma \beta_2 \in I$. 
\item[(b)] Since $A$ is a string algebra if $(\rho, \gamma) \in {^-(0,0)_{n}^+}$ there exists a unique arrow $\beta$ such that $\gamma \beta \in \P$ and then $\rho \beta \in AP_{n+1}$. It is clear that $(\rho \beta, \gamma \beta) \in \!^{-}(0,1)_{n+1} \cap
( AP_{n+1} // \P_2) $  and $F^1_{n+1} ((\rho, \gamma))= (-1)^{n+1}(\rho \beta, \gamma \beta)$.
\item[(c)] Analogous to the previous one.
\item[(d)] By construction it is clear that the map $\phi_n$ is a bijection.  The first statements follow from the fact that if $(\alpha\rho', \alpha\gamma') \in (1,0)_n^+$, then $\phi_n (\alpha \rho', \alpha \gamma') = (\rho' \beta, \gamma' \beta)$,  
 $(\alpha \rho' \beta, \alpha \gamma' \beta) \in {(1,1)_{n+1}} \cap ( AP_{n+1} // \P_2)$ and 
$$F^1_{n+1}( \rho' \beta, \gamma' \beta) = (\alpha\rho'\beta, \alpha \gamma' \beta) =(-1)^{n+1} F^1_{n+1} (\alpha \rho', \alpha  \gamma').$$ 
Finally if $(\rho, \gamma ) \in {^+(0,0)^+_{n} \cap ( AP_n // \P_1 )}$ there exist unique arrows $\alpha, \beta$ such that $\alpha\gamma \beta \in \P$. So  $(\alpha\rho, \alpha \gamma) \in (1,0)^+_{n+1}  \cap ( AP_{n+1} // \P_2)$, $\phi_{n+1} (\alpha \rho, \alpha \gamma) = ( \rho \beta,\gamma \beta) \in {^+(0,1)_{n+1}}\cap ( AP_{n+1} // \P_2)$ and 
\[F_{n+1}((\rho, \gamma )) = (\alpha \rho, \alpha \gamma)+(-1)^{n+1} (\rho \beta,\gamma \beta). \] 
\end{itemize} 
\end{proof}

\medskip

\begin{proposition} \label{nuc im}
For any $n \geq 1$ we have that
\begin{align*}
\dim \Ker F^1_{n+1} & =  \vert ^-(0,0)_{n}^- \vert  +  \vert  ^-(0,1)_{n} \vert  + \vert   (1,1)_{n} \vert  +
 \vert  (1,0)_{n} \vert, \\
\dim \Im F^1_{n+1} & =   \vert \left ( ^{-}(0,1)_{n+1} \sqcup  (1,0)_{n+1} \sqcup  (1,1)_{n+1}  \right )  \cap  ( AP_{n+1} // \P_2)  \vert .
\end{align*}
\end{proposition}

\begin{proof}
It follows directly since the previous lemma implies that
the set $$^-(0,0)_{n}^- \sqcup (1,0)_{n}^- \sqcup \ ^-(0,1)_{n} \sqcup  (1,1)_{n}  \sqcup (id +(-1)^{n} \phi_{n}) ((1,0)_{n}^+)$$ is a basis of $\Ker F^1_{n+1}$,
the set $$\left  ( ^{-}(0,1)_{n+1}  \sqcup (1,0)^{-}_{n+1} \sqcup  (1,1)_{n+1} \sqcup  (id  + (-1)^{n+1} \phi_{n+1}) ( (1,0)^+_{n+1}) \right  ) \cap k ( AP_{n+1} // \P_2) $$  is a basis of $\Im F^1_{n+1}$ and
$$ (1,0)_{n} = (1,0)^+_{n} \sqcup (1,0)^-_{n}.$$ 
\end{proof}

The following result will be used in the description of the cup product defined in $\HH^*(A)$. For any cocycle $f$ in $\Ker F_{n+1}$ we denote by $\overline{f}$ its equivalence class in $\HH^n(A)$.

\begin{proposition} \label{especial} Let $n >1$ and let $$f= \sum_{i=1}^ m \lambda_i (\rho_i, \gamma_i) \in \Ker F^1_{n+1} \cap k(AP_n // \P_2)$$ such that $(\rho_i, \gamma_i) \not \in  {^-(0,0)_{n}^-}$ for all $i$ with $1 \leq i \leq m$.   Then $\overline{f}= 0$ in
$\HH^{n}(A)$.
\end{proposition}

\begin{proof}
From Lemma \ref{nucleo} we know that $f$ is a linear combination of
elements in  $$\left ( (1,0)_{n}^- \sqcup   {^-(0,1)_{n}} \sqcup
(1,1)_{n} \sqcup (id +(-1)^{n} \phi_{n})((1,0)_{n}^+) \right  ) \cap k
(AP_{n} // \P_2).$$ Now,
\begin{align*}
(1,0)_{n}^-  \cap  (AP_{n} // \P_2) & =    F^1_{n} (^+(0,0)_{n-1}^- )  \\
 {^-(0,1)_{n}} \cap (AP_{n} // \P_2) & =    F^1_{n} (^-(0,0)_{n-1}^+ ) \\
  (1,1)_{n} \cap  (AP_{n} // \P_2) & =      (-1)^{n} F^1_{n} ((1,0)_{n-1}^+) \\
(id +(-1)^{n} \phi_{n})((1,0)_{n}^+)  \cap k (AP_{n} // \P_2) & =  
F^1_{n}( ^+(0,0)^+_{n-1} \cap ( AP_{n-1} // \P_1) )
\end{align*}
and these equalities imply that $f$ belongs to the image of $F_n$. 
\end{proof}

In order to describe the behavior of $F^0_{n+1}$ in $k(AP_{n}//Q_0)$ we need a description of the basis elements of this vector space. Observe that $(AP_{n}//Q_0) \subset {^+(0,0)_{n}^+ }$. \\

\begin{definition}
A pair $(\alpha_1 \cdots \alpha_n , e_{s(\alpha_1)} ) \in (AP_{n}//Q_0)$ is called { \bf incomplete}  if $\alpha_n \alpha_1 \not  \in I$, and it is called {\bf complete} if $\alpha_n \alpha_1 \in I$.
We denote by $\I _n$ and by $\C_n$ the set of incomplete and complete pairs in $(AP_{n}//Q_0)$, respectively.
\end{definition}
The cyclic group $\Z_n = <\tau>$ of order $n$ acts on $\C_n$, with the action given by
$$ \tau (\alpha_1 \cdots \alpha_n, e_{s(\alpha_1)}) = (\alpha_n \alpha_1 \cdots \alpha_{n-1}, e_{s(\alpha_n)}).$$

For any $(\rho, e) \in \C_n$ we define its {\bf order} as the first natural number  $r$ such that $\tau^r (\rho, e) = (\rho, e)$, and consider the norm of this element defined as follows
$$N( \rho, e) = \sum_{i=0}^{r-1} \tau^i (\rho, e).$$
Clearly $\tau$ and $N$ induce linear maps $\tau,N : k \C_n \to k\C_n$.
We consider the following subset of complete pairs
\begin{multline*}
\C_n(0) =  \{ (\alpha_1 \cdots \alpha_n , e_{s(\alpha_1)}) \in \C_n : \\
 \mbox{ $\nexists \gamma \in Q_1 \setminus \{\alpha_n \}$ and $\nexists \beta \in Q_1 \setminus \{\alpha_1 \}$ with
$\alpha_n \beta, \gamma \alpha_1  \in I$} \}.
\end{multline*}

\begin{definition}\label{defpargentil}
\begin{itemize}
\item [ ]
\item [(a)] A complete pair $(\alpha_1 \cdots \alpha_n , e_{s(\alpha_1)}) $ is called  { \bf gentle } if $ \tau^m(\alpha_1 \cdots \alpha_n, e_{s(\alpha_1)}) \in \C_n (0)$ for any $m \in \Z$;
\item [(b)] An incomplete pair
$(\alpha_1 \cdots \alpha_n , e_{s(\alpha_1)}) $  is called {\bf empty} if there is no relation $\beta \gamma \in I$ with $t(\beta) = s(\alpha_1) = s(\gamma)$.
\end{itemize}
We denote $\G _n$ and $\E_n$ the set of gentle and empty pairs in $(AP_{n}//Q_0)$ respectively, and $\nG_n, \nE_n$ their corresponding complements in $\C_n$ and $\I_n$ respectively, that is
$$(AP_{n}//Q_0) = \C_n \sqcup \I_n, \qquad \C_n = \G_n \sqcup \nG_n, \qquad \I_n = \E_n \sqcup \nE_n.$$
\end{definition}

 \begin{example}  For the algebra presented in Example \ref{ejemplo1} we have that, for any $n>0$,
\begin{align*}
\nG_n   & =   
\begin{cases} 
\{(\alpha_1\alpha_2)^i, e_1), (\alpha_2\alpha_1)^i, e_2)\} & \mbox{ if $n = 2i$, }\\
\emptyset & \mbox{otherwise,}\end{cases} \\
 \G_n &= \emptyset, \\
\nE_n &=  
\begin{cases}
 \{\beta_1(\alpha_2\alpha_1)^{i-1}\alpha_2, e_1)\} & \mbox{ if $n = 2i$,}\\
\emptyset & \mbox{otherwise,}\end{cases} \\
 \E_n &= \emptyset,\end{align*}
and in   Example \ref{ejemplo2} we have that
\begin{align*} \G_n & =    \begin{cases} (AP_n//Q_0) & \mbox{ if $n = 2i$, }\\
\emptyset, & \mbox{otherwise,}\end{cases} \\
\nG_n & = \emptyset, \\
\E_n & =   \begin{cases} \{(\alpha_1\dots \alpha_7, e_1) \}& \mbox{ if $n=7$,}\\
\emptyset & \mbox{otherwise,}\end{cases} \\
  \nE_n & = \emptyset. \end{align*}

\end{example}

\medskip

In order to describe the map $F^0_{n+1} : k (AP_n // Q_0) \to k(AP_{n+1} // Q_1)$ it is enough to study its behavior in each direct summand $k \G_n, k\nG_n, k\E_n$ and $k\nE_n$.

Observe that any element in the set $(1,1)_{n+1} \cap (AP_{n+1} // Q_1)$ is of the form
$$(\alpha_1 \alpha_2 \cdots \alpha_n \alpha_1, \alpha_1)$$
where $(\alpha_1 \alpha_2 \cdots \alpha_n, e_{s( \alpha_1)})$ is a complete pair in $\C_n$, belonging to either $\G_n$ or $\nG_n$. Hence we can decompose  $(1,1)_{n+1} \cap (AP_{n+1} // Q_1)$ in the disjoint union
$$(1,1)_{n+1}^\G \sqcup (1,1)_{n+1}^{\nG}.$$

\begin{lemma} \label{exacto}
The sequence
\[ \xymatrix{  k \G_{n} \ar[r]^N &  k \G_{n} \ar[r]^{(1-\tau)}  &   k \G_{n} \ar[r]^N &   k \G_{n}
}\] is exact.
\end{lemma}

\begin{proof}
It is clear that $ N (1-\tau) = (1-\tau) N = 0$.  Any $x \in  k \G_{n}$ can be written as follows
$$x = \sum_i \sum_{j=0}^{m_i-1} \lambda_{ij} \tau^j (\rho_i, e_{s_i}),$$
with $m_i$ the order of $(\rho_i, e_{s_i})$,  $\lambda_{ij} \in k$ and  $ (\rho_k, e_{s_k}) \not =  \tau^j (\rho_i, e_{s_i})$ if $k \not = i$. Now, $N(x) =0$ implies that
$$0 =  \sum_{i, j } \lambda_{ij} N  \tau^j (\rho_i, e_{s_i}) =  \sum_{i, j} \lambda_{ij}N (\rho_i, e_{s_i})$$
because $N \tau = N$. This implies that $\sum_{j =0}^{m_i-1}  \lambda_{ij} = 0$, and hence
$$x = \sum_{i, j } \lambda_{ij} \tau^j (\rho_i, e_{s_i}) - \sum_i (\sum_{j =0}^{m_i-1} \lambda_{ij} ) (\rho_i, e_{s_i}) = \sum_{i} \sum_{j =1}^{m_i-1}  \lambda_{ij} (\tau^j -1)  (\rho_i, e_{s_i}) \in \Im (1-\tau).$$
On the other hand, if $(1-\tau) x = 0 $ then
$$   \sum_{i} \sum_{j =0}^{m_i-1}   \lambda_{ij} \tau^j (\rho_i, e_{s_i}) =    \sum_{i} \sum_{j =1}^{m_i} \lambda_{i(j-1)} \tau^{j} (\rho_i, e_{s_i})$$
so $\lambda_{i0}= \lambda_{ij}$ for any $j$, and hence
$$x =  \sum_{i} \sum_{j =0}^{m_i-1}   \lambda_{i0} \tau^j (\rho_i, e_{s_i}) = \sum_i  \lambda_{i0} N (\rho_i, e_{s_i}) \in \Im N.$$ 
\end{proof}

\begin{lemma} \label{cero}
If $n \geq 1$ then $F^0_{n+1} = G^1_{n+1} \oplus G^2_{n+1} \oplus G^3_{n+1}$ with
\begin{align*}
G^1_{n+1} & :   k \E_n \to 0 ,\\
G^2_{n+1} & :  k \G_n \to k(1,1)_{n+1}^\G ,\\
G^3_{n+1} & :  k ( \nE_n \sqcup   \nG_n) \to k(1,1)_{n+1}^{\nG} \oplus k \left(  ( (1,0)_{n+1} \sqcup  (0,1)_{n+1} ) \cap (AP_{n+1} // Q_1)   \right ).
\end{align*} 
Moreover, $G^3_{n+1}$ is injective, $\Ker G^2_{2m+1} = k \G_{2m}/\Im (1-\tau)$ and $$\Ker G^2_{2m}= 
\begin{cases} 
 0 \qquad & \mbox{ if $\car k \not = 2$},\\
k \G_{2m-1}/\Im (1-\tau) &\mbox{ if $\car k  = 2$.} \end{cases}$$
\end{lemma}

\begin{proof} 
Recall that
\begin{multline*}
F^0_{n+1} (\alpha_1 \cdots \alpha_{n},  e_{s(\alpha_1)})  \\
= \sum_{\substack{ \{ \beta \in Q_1:    \beta \alpha_1 \in I,\\   t(\beta)= s(\alpha_1) \} }} (\beta \alpha_1 \cdots \alpha_{n}, \beta) + (-1)^{n+1} \sum_{ \substack{ \{ \beta \in Q_1:  \alpha_{n} \beta \in I , \\   s(\beta)= s(\alpha_1)  \} } } (\alpha_1 \cdots \alpha_{n} \beta, \beta )  \end{multline*}
so it is clear that $F^0_{n+1} (\E_n ) =0$, and if $(\alpha_1 \cdots \alpha_n , e_{s(\alpha_1)}) \in \G_n$ then
$$G^2_{n+1} (\alpha_1 \cdots \alpha_{n},  e_{s(\alpha_1)}) =  (\alpha_n \alpha_1 \cdots \alpha_{n}, \alpha_n) + (-1)^{n+1} (\alpha_1 \cdots \alpha_{n} \alpha_1, \alpha_1) \in k(1,1)_{n+1}^\G.$$
Now we shall prove that $\Ker G^2_{2m+1} = k \G_{2m}/\Im (1-\tau) $.  From the commutativity of the diagram
\[ \xymatrix{  k \G_{2m} \ar[rr]^{ G^2_{2m+1}} \ar@{=}[d] & &  k (1,1)_{2m+1}^\G  \\
 k \G_{2m} \ar[rr]_{(1-\tau)} & &  k \G_{2m} \ar[u]^{-\psi}
}\]
where $\psi  (\alpha_1 \cdots \alpha_{2m}, e_{s(\alpha_1)}) =  (\alpha_1 \cdots \alpha_{2m} \alpha_1,\alpha_1)$ is an isomorphism, we get that $$\Ker G^2_{2m+1} = \Ker (1-\tau).$$
The equalities
$$\Ker (1-\tau) = \Im N = k \G_{2m}/\Ker N = k \G_{2m}/\Im (1-\tau) $$  follow from the exactness of the sequence
\[ \xymatrix{  k \G_{2m} \ar[r]^N &  k \G_{2m} \ar[r]^{(1-\tau)}  &   k \G_{2m} \ar[r]^N &   k \G_{2m},
}\]
proved in Lemma \ref{exacto}.  If $\car k = 2$ the same proof works for $\Ker G^2_{2m}$.  
If $\car k \not = 2$,  we define $T: k(1,1)_{2m}^\G \to k \G_{2m-1}$ as follows:
$$T (\alpha_1 \cdots \alpha_{2m-1} \alpha_1, \alpha_1) = \frac 12 \sum_{i=0}^{2m-2} (-1)^i \tau^i (\alpha_1 \cdots \alpha_{2m-1}, e_{s(\alpha_1)})$$
and we get that $T \circ G^2_{2m} =  \id$ and $ G^2_{2m} \circ T =  \id$, hence $G^2_{2m}$ is bijective. 

If $(\alpha_1 \cdots \alpha_n , e_{s(\alpha_1)}) \in \nE_n$ then
\begin{multline*}
G^3_{n+1} (\alpha_1 \cdots \alpha_{n}, e_{s(\alpha_1)}) \\
=
\begin{cases}
 (\beta \alpha_1 \cdots \alpha_{n}, \beta) & \in k(1,0)_{n+1},  \\
 (-1)^{n+1} (\alpha_1 \cdots \alpha_{n} \gamma, \gamma )  & \in  k(0,1)_{n+1}, \\
 (\beta \alpha_1 \cdots \alpha_{n}, \beta) +
 (-1)^{n+1} (\alpha_1 \cdots \alpha_{n} \gamma, \gamma ) & \in k((1,0)_{n+1} \sqcup  (0,1)_{n+1})
\end{cases} \end{multline*}
depending on the existence of $\beta$ and $\gamma$ satisfying $\beta \not = \alpha_n$ and $\beta \alpha_1 \in I$,  $\gamma \not = \alpha_1$ and $\alpha_n \gamma \in I$.
Finally, if $(\alpha_1 \cdots \alpha_n , e_{s(\alpha_1)}) \in \nG_n$ then
\begin{align*}   G&^3_{n+1}  (\alpha_1 \cdots \alpha_{n}, e_{s(\alpha_1)})  \\
& =   \begin{cases} (\alpha_n \alpha_1 \cdots \alpha_{n}, \alpha_n)
+ (-1)^{n+1} (\alpha_1 \cdots \alpha_{n} \alpha_1, \alpha_1),
\\ (\alpha_n \alpha_1 \cdots \alpha_{n}, \alpha_n) + (-1)^{n+1}
(\alpha_1 \cdots \alpha_{n} \alpha_1, \alpha_1) +
 (\beta \alpha_1 \cdots \alpha_{n}, \beta),  \\
 (\alpha_n \alpha_1 \cdots \alpha_{n}, \alpha_n) + (-1)^{n+1} (\alpha_1 \cdots \alpha_{n} \alpha_1, \alpha_1) + (-1)^{n+1} (\alpha_1 \cdots \alpha_{n} \gamma, \gamma ), \\
 (\alpha_n \alpha_1 \cdots \alpha_{n}, \alpha_n) + (-1)^{n+1} (\alpha_1 \cdots \alpha_{n} \alpha_1, \alpha_1) + (\beta \alpha_1 \cdots \alpha_{n}, \beta) +
 (-1)^{n+1} (\alpha_1 \cdots \alpha_{n} \gamma, \gamma )
\end{cases}
\end{align*}
depending on the existence of $\beta$ and $\gamma$ in $Q_1$ satisfying  $\beta \not = \alpha_n$ and $\beta \alpha_1 \in I$,   $\gamma \not = \alpha_1$ and $\alpha_n \gamma \in I$. Hence 
\begin{align*} G^3_{n+1} (k \nE_n)  & \subset   
 k\left ( ((1,0)_{n+1} \sqcup  (0,1)_{n+1}) \cap (AP_{n+1} //Q_1) \right ), \\
 G^3_{n+1} (k  \nG_n) & \subset  
k\left ( ( (1,1)_{n+1}^{\nG} \sqcup (1,0)_{n+1} \sqcup  (0,1)_{n+1}) \cap  (AP_{n+1} //Q_1) \right ).
\end{align*} 
Now we define the linear map $$T: k \left ( ( (1,1)_{n+1}^{\nG} \sqcup (1,0)_{n+1} \sqcup  (0,1)_{n+1}) \cap  (AP_{n+1} //Q_1) \right ) \to k ( \nE_n \sqcup   \nG_n)$$
as follows: 
\begin{align*}
T (\alpha_1 \cdots \alpha_n \alpha_1 , \alpha_1) & =  (-1)^{n+1} \sum_{i=0}^{\mu(w)-1}(-1)^{in} \tau^{i}(  w ), \\
T(\beta \alpha_1 \cdots \alpha_{n}, \beta) & =  \begin{cases}
w  & \mbox{if $\alpha_n \alpha_1 \not \in I$}, \\
0 & \mbox{if $\alpha_n \alpha_1 \in I$ and $\exists \gamma \not = \alpha_1$ } \\ & \mbox{ such that $\alpha_n \gamma \in I$}, \\
 (-1)^n \sum_{i=0}^{\mu(\tau(w))-1}(-1)^{in}\tau^{i+1}(w) &
\mbox{otherwise},
\end{cases} \\
T( \alpha_1 \cdots \alpha_{n} \gamma, \gamma) & =   \begin{cases}
 - \sum_{i=0}^{\mu(\tau(w))-1}(-1)^{in} \tau^{i+1}(w)  & \mbox{if $\alpha_n \alpha_1 \in I$} ,\\
0 & \mbox{if $\alpha_n \alpha_1 \not \in I$ and $\exists \beta \not = \alpha_n$ } \\ & \mbox{  such that $\beta \alpha_1 \in I$}, \\
(-1)^{n+1}w  & \mbox{otherwise}
\end{cases}
\end{align*}
where $w= (\alpha_1 \cdots \alpha_n, e_{s(\alpha_1)})$ and $\mu (w)$ is the first natural number such that $\tau^{\mu(w)-1}(w) \not \in \C_n(0)$. 
A direct computation shows that $T \circ G^3_{n+1} =  \id$, and hence $G^3_{n+1}$ is injective. 
\end{proof}

\medskip

\begin{proposition}  \label{cero final}
If $n \geq 1$  then
\begin{align*}
\dim \Ker F^0_{n+1} &  = 
\begin{cases}  \vert \E_{n} \vert + \dim k \G_{n} / \Im (1-\tau)  & \quad \mbox{if $n$ is even and $\car k \not = 2$,} \\
\vert \E_{n} \vert  & \quad \mbox{if $n$ is odd and $\car k \not = 2$,} \\
\vert \E_{n} \vert + \dim k \G_{n} / \Im (1-\tau)  & \quad \mbox{if $\car k  = 2$,} 
\end{cases} \\
\dim \Im F^0_{n+1} & = 
\begin{cases} \vert \C_{n} \vert + \vert  \nE_{n} \vert -  \dim k \G_{n}/ \Im (1-\tau) & \quad \mbox{if $n$ is even and $\car k \not = 2$,} \\
\vert \C_{n} \vert  + \vert \nE_{n} \vert & \quad \mbox{if $n$ is odd and $\car k \not = 2$,} \\
\vert \C_{n} \vert + \vert  \nE_{n} \vert -  \dim k \G_{n}/ \Im (1-\tau) & \quad \mbox{if $\car k  = 2$.}
\end{cases}
\end{align*}
\end{proposition}

\begin{proof}
The formula for the dimension of $ \Ker F^0_{n+1}$ follows by a direct computation using Lemma \ref{cero}. The equalities
$$ \dim \Im F^0_{n+1} = \vert (AP_n // Q_0) \vert - \dim \Ker F^0_{n+1}$$
and
$$\vert (AP_n // Q_0) \vert = \vert \G_n  \vert + \vert \nG_n \vert + \vert \E_n \vert + \vert \nE_n \vert = \vert \C_n \vert + \vert \E_n \vert + \vert \nE_n \vert$$
imply the formula for $ \dim \Im F^0_{n+1}$. 
\end{proof}

\subsection{Dimensions of the Hochschild cohomology groups}

We start this section with the computation of the first Hochschild cohomology groups.

\medskip

\begin{theorem} \label{0-1} Let $A = kQ/I$ be a quadratic string  algebra.  Then
\begin{align*}
\dim \HH^0(A) & =   \vert  ^-(Q_0 // \P_1)^-  \vert \  + 1,\\
\dim \HH^1(A) & =    \begin{cases} \vert  ^-(0,0)^-_1  \vert  +  \vert   Q_1
\vert  - \vert   Q_0  \vert \ + 1  & \mbox{if $\car k \not = 2$},\\
\vert  ^-(0,0)^-_1  \vert  +  \vert   Q_1
\vert  - \vert   Q_0  \vert \ + 1 + \vert \G_1\vert & \mbox{if $\car k  = 2$.}
\end{cases}
\end{align*}
\end{theorem}

\begin{proof}
Recall that the maps $ F_1^0$ and $F_1^1$ appearing in the commutative diagram
\[ \xymatrix{
  \Hom_{E-E}(  k Q_0, A) \ar[r]^-{F_{1}} \ar^\cong[dd] &   \Hom_{E-E}(  k Q_{1}, A) \ar^\cong[dd] &  \\ \\
   k(Q_0// Q_0)\oplus k(Q_0// \mathcal{P}_1)  \   \ar[r]^-{  \left ( \begin{matrix} 0 & 0 \\ F^0_{1}  & 0 \\ 0 & F^1_{1}
\end{matrix} \right )} \   &  \ k(Q_{1} // Q_0) \oplus k(Q_{1} // Q_1)\oplus k(Q_{1} // \mathcal{P}_2)  } \]
are given by
\[ F_1^0 (e_r, e_r) =   \sum_{ \{ \beta \in Q_1:  \ t(\beta)=r \} } (\beta , \beta)  - \sum_{ \{ \beta \in Q_1: \ s(\beta)=r  \} } (\beta, \beta ),     \]
\[ F_1^1 (e_r, \gamma) =   \sum_{ \{ \beta \in Q_1:  \ t(\beta)=r \} } (\beta , \beta \gamma)  - \sum_{ \{ \beta \in Q_1: \ s(\beta)=r  \} } (\beta, \gamma \beta ).    \]
Then
$$F_1^0 ( \sum_{ r  \in Q_0}   \lambda_r (e_r, e_r))=  \sum_{ \beta \in Q_1 } ( \lambda_{t(\beta)} - \lambda_{s(\beta)} ) (\beta , \beta) =0$$
implies that $\lambda_i = \lambda_j$ whenever there exists an arrow $\beta: i \to j$.  Since $Q$ is connected, we have that $\lambda_i = \lambda_j$ for any $i,j$.  Hence
$\dim \Ker F_1^0= 1$.

On the other hand,  $(Q_0 // \P_1) = ^-(Q_0 // \P_1)^- \sqcup ^-(Q_0 // \P_1)^+ \sqcup ^+(Q_0 // \P_1)^- \sqcup ^+(Q_0 // \P_1)^+$ and we have that:
\begin{itemize}
\item [(i)] $F_1^1 (^-(Q_0 // \P_1)^-) = 0$;
\item [(ii)] $F_1^1$ induces a bijection from $^-(Q_0 // \P_1)^+$ to $^{-}(0,1)_{1}$;
\item[(iii)]  $F^1_{1}$ induces a bijection from $^+(Q_0 // \P_1)^-$ to $(1,0)^{-}_{1}$;
\item[(iv)] there exists a bijection $\phi_1: (1,0)_{1}^+ \to {^+(0,1)_{1}}$ given by $\phi_1 (\alpha , \alpha  \gamma) = ( \beta,  \gamma \beta)$ such that
$F^1_{1} (^+(Q_0 // \P_1)^+)  =  (id  - \phi_{1}) ((1,0)^+_{1}).$
\end{itemize}
Hence $^-(Q_0 // \P_1)^-$ is a basis for $\Ker F_1^1$ and then
$$  \dim  \HH^0(A)  = \dim \Ker F_1 = \dim \Ker F_1^0 + \dim \Ker F_1^1 =  1 +   \vert ^-(Q_0 // \P_1)^- \vert .$$
From these computations we get that
\begin{align*}
\dim \Im F_1^0  & =  \vert (Q_0 // Q_0) \vert - \dim \Ker F_1^0 = \vert Q_0 \vert -1 , \\
 \dim \Im F_1^1 & =  \vert (Q_0 // \P_1) \vert - \dim \Ker F_1^1 \\
& =  \vert ^-(0,1)_1 \vert + \vert
(1,0)^-_1 \vert + \vert (1,0)^+_1  \vert  .
 \end{align*}
From Proposition \ref{nuc im} and Proposition \ref{cero final} we have that
\begin{align*}
 \dim \Ker F_2 & =  \dim \Ker F_2^0 + \dim \Ker F_2^1 \\
 & =   \vert \E_1 \vert  + \vert ^-(0,0)^-_1\vert +  \vert ^-(0,1)_1
\vert + \vert
(1,1)_1\vert + \vert (1,0)_1 \vert 
\end{align*}
if $\car k \not = 2$ and 
$$\dim \Ker F_2 =   \vert \E_1 \vert  + \vert ^-(0,0)^-_1\vert +  \vert ^-(0,1)_1 \vert + \vert
(1,1)_1\vert + \vert (1,0)_1 \vert + \dim k\G_1/\Im (1-\tau)$$
if $\car k = 2$. 
Now $\vert (1,1)_1 \vert  = \vert  Q_1 \vert $,  $\E_1 = \emptyset$ since $A$ is finite dimensional and 
$$k\G_1/\Im (1-\tau) = k\G_1 ,$$
so
\[\dim \HH^1(A) =  \begin{cases} \vert  ^-(0,0)^-_1  \vert  +  \vert   Q_1
\vert  - \vert   Q_0  \vert \ + 1  & \mbox{if $\car k \not = 2$},\\
\vert  ^-(0,0)^-_1  \vert  +  \vert   Q_1
\vert  - \vert   Q_0  \vert \ + 1 + \vert \G_1 \vert & \mbox{if $\car k  = 2$.}
\end{cases}\] 
\end{proof}

The dimensions that have been computed in Propositions \ref{nuc im} and \ref{cero final} lead us to the following theorem.
\begin{theorem}\label{n>1}
If $A = kQ/I$ is a quadratic string  algebra and $n \geq 2$   then
\begin{align*}
\dim \HH^{n}(A)  =   \vert  ^-(0,0)_n^- \vert + \vert  \E_n  \vert    - \vert \nE_{n-1} \vert  + \vert  \left ( (1,0)_n \sqcup
{^-(0,1)}_n \right ) \cap (AP_n//Q_1)  \vert \\
 +        \begin{cases}
\dim  k\G_n/\Im(1-\tau)     \quad & \mbox {if $n$ is even and $\car k \not = 2$}, \\
\dim k\G_{n-1}/\Im (1-\tau)
   & \mbox{if $n$ is odd and $\car k \not = 2$,} \\
\dim  k\G_n/\Im(1-\tau)  + \dim k\G_{n-1}/\Im (1-\tau)   \quad & \mbox {if $\car k  = 2$.}
\end{cases} \end{align*}
\end{theorem}

\begin{proof}
From Propositions \ref{cero final} and \ref{nuc im} we get that if $n$ is even and $\car k \not = 2$ then
\begin{align*}
\dim \HH^{n}(A) &= \dim \Ker F_{n+1} - \dim \Im F_{n} \\
&=    \vert ^-(0,0)_{n}^- \vert  +  \vert  ^-(0,1)_{n} \vert  + \vert   (1,1)_{n} \vert  +
 \vert  (1,0)_{n} \vert \\
 & -   \vert \left ( ^{-}(0,1)_{n} \sqcup  (1,0)_{n} \sqcup  (1,1)_{n}  \right )  \cap  ( AP_{n} // \P_2)  \vert  \\
& +    \vert \E_{n} \vert + \dim k \G_{n} / \Im (1-\tau)
- \vert \C_{n-1} \vert  - \vert \nE_{n-1} \vert
\end{align*}
and the desired formula follows since the identification
$$(\alpha_1 \cdots \alpha_{n-1} \alpha_1, \alpha_1) \leftrightarrow (\alpha_1 \cdots \alpha_{n-1}, e_{s(\alpha_1)})$$
implies  that $\vert (1,1)_n \cap ( AP_{n} // Q_1) \vert 
= \vert  \C_{n-1} \vert $.   
Similarly one can deduce the other formulae. 
\end{proof}

\begin{definition} \label{gentle}
A string algebra $A = kQ/I$ is called a { \bf gentle algebra} if in
addition  $(Q,I)$ satisfies :
\begin{itemize}
\item [G1)]  For an arrow $\alpha$ in $Q$ there exists  at most one arrow $\beta$ and at most one arrow $\gamma$ such that $\alpha\beta  \in I$ and $ \gamma \alpha  \in I$;
\item [G2)] $I$ is quadratic.
\end{itemize}
\end{definition}

The Hochschild cohomology groups of gentle algebras have already been computed in~\cite{L}, and these results have been expressed in terms of the derived invariant introduced by Avella-Alaminos and Geiss in~\cite{AAG}.   As a consequence of our previous theorem, we recover the results in~\cite{L}.

\begin{corollary}\label{gentil}
If $A = kQ/I$  is a gentle algebra, then 
\begin{align*}
\dim \HH^0 (A) & =    \vert  ^-(Q_0 // \P_1)^-  \vert  + 1, \\
\dim \HH^1 (A) & = 
\begin{cases} \vert  ^-(0,0)^-_1  \vert  +  \vert   Q_1
\vert  - \vert   Q_0  \vert \ + 1  & \mbox{if $\car k \not = 2$},\\
\vert  ^-(0,0)^-_1  \vert  +  \vert   Q_1
\vert  - \vert   Q_0  \vert \ + 1 + \vert (Q_1 //Q_0) \vert & \mbox{if $\car k  = 2$,}
\end{cases} \\
\dim \HH^n (A) & =  \vert {^-(0,0)^-_n} \vert +  \vert   \E_n \vert + a \dim k\G_n/\Im(1-\tau) + b\dim  \G_{n-1}/\Im(1-\tau) 
\end{align*}
where
\[ (a,b)= 
\begin{cases}
(1,0)   & \mbox{if $n \geq 2$, $n$  even, $\car k \not = 2$},\\
(0,1) & \mbox{if $n \geq 2$,  $n$ odd, $\car k \not = 2$,} \\
(1,1)  & \mbox{if $n \geq 2$, $\car k  = 2$. }
\end{cases} \]
\end{corollary}

\begin{proof} From Theorems \ref{0-1} and \ref{n>1} it is clear that we only have to prove that 
$$ \vert \nE_{n-1} \vert  =     \vert  \left ( (1,0)_n \sqcup
{^-(0,1)}_n \right ) \cap (AP_n//Q_1)  \vert .$$
Since $A$ is gentle,  $\nG_{n-1} = \emptyset$, and in this case the injective map  
$$G^3_n: k (\nE_{n-1} ) \to k  \left ( ( (1,0)_n \sqcup
{(0,1)}_n  ) \cap (AP_n//Q_1) \right )$$ studied in Lemma \ref{cero}  satisfies
\begin{multline*}
 G^3_{n} (\alpha_1 \cdots \alpha_{n-1}, e_r) \\ =
\begin{cases}
 (\beta \alpha_1 \cdots \alpha_{n-1}, \beta) & \in k((1,0)^-_{n}) , \\
 (-1)^{n} (\alpha_1 \cdots \alpha_{n-1} \gamma, \gamma )  & \in  k({^-(0,1)}_{n}), \\
 (\beta \alpha_1 \cdots \alpha_{n-1}, \beta) +
 (-1)^{n} (\alpha_1 \cdots \alpha_{n-1} \gamma, \gamma ) & \in k((id + (-1)^{n} \phi_{n})
(1,0)^+_{n})
\end{cases} \end{multline*}
depending on the existence of $\beta$ and $\gamma$. So
$$ \vert \nE_{n-1} \vert  = \dim \Im G^3_n = \vert  \left ( (1,0)_n \sqcup
{^-(0,1)}_n \right ) \cap (AP_n//Q_1)  \vert.$$  \end{proof}

\begin{example} The algebra presented in Example \ref{ejemplo1} is a quadratic string algebra, then by Theorems \ref{0-1} and  \ref{n>1} we have that
 \[\dim \HH^n(A) =  \begin{cases} 2  \qquad & \mbox{if $n=0$,} \\
3 & \mbox{if $n=1$,} \\ 
1 & \mbox{if $n=2i+1$, $i\geq 1$, } \\
0 & \mbox{otherwise.} 
\end{cases}\] The algebra in Example \ref{ejemplo2} is a gentle algebra, then we can  compute its cohomology using   Corollary \ref{gentil}, 
 \[\dim \HH^n(A) =  \begin{cases}  1 \qquad & \mbox{if $n=0$,} \\
3  & \mbox{if $n=1$,} \\
2 & \mbox{if   $n = 7$,} \\ 
1 & \mbox{otherwise.} 
\end{cases} \] 
 \end{example}

\medskip

The following results will be used in the description of the Lie bracket defined in $\HH^*(A)$.

\begin{proposition}  \label{gentle cero} Let $A$ be a gentle algebra,  $n >1$  and let   $$f= \sum_{i=1}^m \lambda_i (\rho_i, \gamma_i) \in \Ker F^1_{n+1} $$ such that $(\rho_i, \gamma_i) \not \in  {^-(0,0)_{n}^-}$, $(\rho_i, \gamma_i) \not \in  {(1,1)_{n}\cap (AP_n //Q_1)}$ for all $i$ with $1 \leq i \leq m$.  Then $\overline{f}= 0$ in  ${\HH}^{n}(A)$.
\end{proposition}

\begin{proof}
From Lemma \ref{nucleo} we know that  $f$ can be written as $f = f_1
+ f_2$ where $$f_1 \in \Ker F^1_{n+1} \cap k(AP_n
// Q_1) \quad \mbox{and} \quad f_2 \in \Ker F^1_{n+1} \cap k(AP_n //
\P_2).$$ By Proposition \ref{especial}
 we have that $\overline{f_2} = 0$ in ${\HH}^{n}(A)$.
On the other hand, $f_1$ is a linear combination of elements in
$$\left((1,0)_{n}^- \sqcup {^-(0,1)_{n}}  \sqcup (id +(-1)^{n}
\phi_{n})((1,0)_{n}^+) \right)\cap k(AP_{n} // Q_1).$$ The proof of Corollary \ref{gentil} shows that 
$$ \left(   (1,0)_{n}^- \sqcup ^-(0,1)_{n} \sqcup  (id +(-1)^{n}
\phi_{n})(1,0)_{n}^+ \right) \cap k(AP_{n} // Q_1) = \Im G^3_n \subset \Im F^0_n.$$ Hence
$\overline{f_1} = 0 $ in  ${\HH}^{n}(A)$. 
\end{proof}

\begin{corollary} \label{gentil final}  Let  $A$ be a gentle algebra,  $n >1$ and let  $$f= \sum_{i=1}^m \lambda_i (\rho_i, \gamma_i) \in \Ker F^1_{n+1} $$  such that $(\rho_i, \gamma_i) \not \in  {^-(0,0)_{n}^-}$ for all $i$ with $1 \leq i \leq m$. If $\G_{n-1} = \emptyset$ then $\overline{f} =0$ in  ${\HH}^{n}(A)$.
\end{corollary}

\begin{proof} By assumption $\G_{n-1} = \emptyset$ and, since $A$ is gentle,  $\nG_{n-1} = \emptyset$.  Then 
$$ (1,1)_{n}\cap (AP_n//Q_1) = (1,1)_{n}^{\G}\sqcup (1,1)_{n}^{\nG} = \emptyset$$
and hence the desired result follows from the previous proposition. 
\end{proof}

\section{Gerstenhaber algebra}

\subsection{Comparison morphisms}

The cup product and the Lie bracket in the cohomology $\HH^*(A)$ are
induced by operations defined using the bar resolution, and we have made the computations of $\HH^n(A)$ using  Bardzell's resolution.  
In this section we will construct comparison morphisms between the bar resolution and  Bardzell's resolution in order to get formulae for the cup product and the Lie bracket.

We start by recalling the definition of these structures at the level of cochains using the bar resolution $( A^{\otimes^n} , b_n)_{n\geq 0}$.  Given $f \in \Hom_{E-E}(
A^{\otimes^{n}} , A)$ and $g \in \Hom_{E-E}( A^{ \otimes^{m}} , A)$ we have
$$f \cup g \in \Hom_{E-E}( A^{\otimes^{m+n}} , A) \quad \mbox{and} \quad [f,g] \in \Hom_{E-E}(
A^{\otimes^{m+n-1}} , A)$$ defined by
$$  f \cup g ( v_1 \otimes \cdots \otimes v_{n+m}  ) = f( v_1 \otimes \cdots \otimes v_{n}  )g( v_{n+1} \otimes \cdots \otimes v_{n+m}  ) $$
and $$ [f,g ]= f \circ g - (-1)^{(n-1)(m-1)} g \circ f
$$where
\begin{equation} f \circ g = \sum_{i=1}^n(-1)^{(i-1)(m-1)}f \circ_i g \label{(1)}
\end{equation} and
$$ f \circ_i g(v_1 \otimes \ldots \otimes v_{n+m-1}  ) = f(v_1 \otimes
\ldots \otimes v_{i-1} \otimes g( v_i \otimes \ldots v_{i+m-1} ) \otimes v_{i+m} \otimes
\ldots v_{n+m-1} ). $$
These structures are easily carried to Bardzell's resolution
using the comparison morphisms 
\[ \xymatrix{  
A \otimes kAP_n \otimes A   \ar@<.9ex>[r]^{U_n}&     \ar@<.9ex>[l]^{V_n} A \otimes A^{\otimes ^n} \otimes A
} \]
given by the $A$-bimodule morphisms 
\begin{align*}
U_0(1 \otimes e \otimes 1) & =  e \otimes 1, \\
U_n(1\otimes \alpha_1\alpha_2 \ldots \alpha_n \otimes 1) & =  1\otimes \alpha_1 \otimes  \ldots  \otimes \alpha_n \otimes 1 \quad  \mbox{for $n \geq 1$},\\
V_0 (1 \otimes 1) & =   1 \otimes 1 \otimes 1 = \sum_{i \in Q_0} 1 \otimes e_i \otimes 1, \\
V_1 (1 \otimes \gamma \otimes 1) & =  \begin{cases}
0 \quad & \mbox{if $\vert \gamma \vert =0$},\\
\sum_{i=1}^s\alpha_1 \cdots \alpha_{i-1}\otimes \alpha_i\otimes
\alpha_{i+1} \cdots \alpha_s & \mbox{if $\gamma= \alpha_1 \cdots\alpha_s$},
\end{cases}
\end{align*}
\begin{multline*}
V_n (1\otimes v_1 \otimes \cdots \otimes  v_n \otimes1) \\
 = \begin{cases}
  v_1'\otimes \alpha_1 v_2 \cdots v_{n-1}\alpha_n\otimes v_n'  \quad & \mbox{if $v_1= v'_1 \alpha_1$, $v_n = \alpha_n v'_n$} \\
& \mbox{and  $\alpha_1 v_2 \cdots v_{n-1}\alpha_n \in AP_n $, } \\
 0 & \mbox{otherwise}.
\end{cases}
\end{multline*}

\begin{lemma} \label{comparacion}
The maps $ U=(U_n)_{n \geq 0}$ and $ V=(V)_{ n \geq 0}$ are morphism of complexes and   $ V \circ U = \id.$
\end{lemma}

Using the isomorphisms 
$$\Hom_{A-A}(A \otimes A^{\otimes^{n}} \otimes A, A) \simeq
\Hom_{E-E}(  A^{\otimes^{n}} , A),$$
$$\Hom_{A-A}(A \otimes kAP_{n} \otimes A, A) \simeq
\Hom_{E-E}(kAP_{n} , A)$$  
we get that the morphisms $U$ and $V$  induce quasi-isomorphisms 
\begin{align*} U^{\bullet} & =   \left (
U^{n} : \Hom_{E-E}( A^{\otimes^{n}} , A) \longrightarrow \Hom_{E-E}(
kAP_n , A)\right )_{n\geq 0} \mbox{ and } \\
 V^{\bullet}
& =  \left ( V^{n} : \Hom_{E-E}( kAP_n , A)
 \longrightarrow \Hom_{E-E}( A^{\otimes^{n}} , A)\right )_{n\geq 0}.
\end{align*}
With these quasi-isomorphisms, the cup product and the Lie bracket, which we still denote $\cup$ and $[ \ , \ ]$, can be defined as follows:
given $f \in \Hom_{E-E}(kAP_n , A)$ and $g \in \Hom_{E-E}(kAP_m , A)$,
$$f \cup g \in  \Hom_{E-E}(kAP_{n+ m} , A) \quad \mbox{and} \quad f \circ_i g \in \Hom_{E-E}(kAP_{n+ m-1 } , A),  1 \leq i \leq n $$
are defined by 
$$f\cup g  =  U^{n+m}(  V^n(f) \cup  V^m(g)) \qquad \mbox{and} \qquad  f\circ_i g    =    U^{m+n-1}(  V^n(f) \circ_i V^m(g)).$$
As usual,
\begin{align*}
F_{n+m+1}(f \cup g ) &=  F_{n+1}(f) \cup g + (-1)^n f \cup
F_{m+1}(g)  \\ F_{n+m}([ f,g]) &= [f, F_{m+1}(g) ] + (-1)^{m-1} [F_{n+1}(f),
g]
\end{align*} so the products $\cup$ and $[ - , - ]$ defined in the complex
$(\Hom_{E-E}(kAP_{n } , A), F_n)_{n\geq 0}$ induce  products at the cohomology level. 

\subsection{Formulae for the cup product and the Lie bracket}

The following detailed computations will be useful in order to describe the desired structures:
$$f\cup g( \alpha_1 \cdots \alpha_{n+m} ) =f (\alpha_1 \cdots  \alpha_{n}  )g(\alpha_{n+1} \cdots
\alpha_{n+m} ) $$
and, if $g \in \Hom_{E-E}(kAP_m , A)$ is a basis element, that is, $g$ sends a fix basis element in $AP_m$ to a basis element in $\P$, and it is zero otherwise, then
$$f \circ_1 g(\alpha_1 \ldots \alpha_{n+m-1}) = \mu f( \beta \alpha_{m+1} 
\cdots   \alpha_{n+m-1} )$$
if $g(\alpha_1  \cdots  \alpha_{m}) = \mu \beta \in   \P_1 $ and $\beta \alpha_{m+1} \in I$, 
$$f \circ_n g(\alpha_1 \ldots \alpha_{n+m-1})  =   f(  \alpha_{1} 
\cdots  \alpha_{n-1} \beta ) \mu $$
if $g(\alpha_n  \cdots  \alpha_{n+m-1}) =  \beta \mu \in   \P_1$ and $\alpha_{n-1}\beta \in I $, 
$$f\circ_i g  (  \alpha_1 \cdots \alpha_{n+m-1})   =  f (  \alpha_1 \cdots  \alpha_{i-1} \beta \alpha_{i+m} \cdots \alpha_{n+m-1}) $$
if $g(\alpha_i  \cdots  \alpha_{i+m-1}) =  \beta \in  Q_1$ and $\beta \alpha_{i+m}, \alpha_{i-1} \beta \in I$, 
and it is zero otherwise. \\

Now,  using the identification $\Hom_{E-E}( kAP_{m+n}, A) \simeq
k(AP_{m+n}//\P)$ given by  $$f_{(\rho, \gamma) }\longleftrightarrow
(\rho, \gamma),$$
we get that, given 
$$(\rho, \gamma)  =  (\alpha_1 \ldots \alpha_n, \gamma) \in ( AP_n//  \P) \quad \mbox{and} \quad (\rho', \gamma') = (\beta_1 \ldots \beta_m, \gamma') \in ( AP_m//  \P )$$
we have that
\begin{itemize}
\item if $\alpha_n \beta_1 \in I$ and $\gamma \gamma' \not \in I$ then $(\rho, \gamma)  \cup (\rho', \gamma') = (\rho\rho', \gamma\gamma')$;
\item if $\gamma' \in Q_1$, $\gamma' = \alpha_i$,  $\alpha_{i-1} \beta_1 \in I$ and $\beta_m \alpha_{i+1} \in I$ then  
$$(\rho, \gamma)  \circ_i (\rho', \gamma') = (\alpha_1 \dots \alpha_{i-1} \rho' \alpha_{i+1} \dots \alpha_n, \gamma);$$ 
\item if $\gamma' \in  \P_2$, $\gamma'= \mu \alpha_1$ and $\beta_m \alpha_2 \in I$ then $(\rho, \gamma)  \circ_1 (\rho', \gamma')= 
(\rho' \alpha_{2} \cdots \alpha_n, \mu \gamma)$;
\item  if $\gamma' \in  \P_2$, $\gamma'= \alpha_n \mu$ and $\alpha_{n-1} \beta_1 \in I$ then $(\rho, \gamma)  \circ_n (\rho', \gamma') = 
(\alpha_1 \cdots \alpha_{n-1}\rho', \gamma \mu)$
\end{itemize}
and all the other cases are zero.

\subsection{Vanishing of the cup product and the Lie bracket}

In this section we find conditions on the presentation $(Q,I)$ of the quadratic string algebra that ensures the vanishing of the cup product and the Lie bracket.  We start with a remark that will be used throughout this section.

\begin{remark}
Any  $\overline{f}  =  \overline{\sum_i \lambda_i (\rho_i,\gamma_i)
}$ in $ \HH^n(A)$ can be written as $\overline{f} = \overline{f_1} +
\overline{f_2}$ with
$$f_1 = \sum_{\vert \gamma_i \vert = 0}
\lambda_i (\rho_i,\gamma_i) \in
\Ker F_{n+1}^0, \qquad
f_2 =  \sum_{\vert \gamma_i \vert
> 0} \lambda_i (\rho_i,\gamma_i) \in \Ker F_{n+1}^1$$
Moreover from Lemmas \ref{nucleo} and \ref{cero} we have that $f_1 \in k(\E_n \sqcup \G_n)$
and $$f_2 \in k(^-(0,0)_n^- \sqcup (1,0)_n^-
\sqcup ^-(0,1)_n \sqcup (1,1)_n \sqcup (id +(-1)^{n}
\phi_{n})((1,0)_{n}^+)) .$$
\end{remark}

\begin{proposition} 
Let $A = k Q/I$ be a quadratic string algebra, $n, m >0$.  If $\mathcal{G}_n = \emptyset = \mathcal{G}_m$ then
$\HH^n(A) \cup \HH^m(A) = 0$.
\end{proposition}

\begin{proof}
Let $\overline{f} = \overline{ f_1} + \overline{f_2} \in \HH^n(A)$, $\overline{g} =  \overline{g_1} + \overline{g_2} \in \HH^m(A)$ with $f_1 \in \Ker F_{n+1}^0, f_2 \in \Ker F_{n+1}^1, g_1 \in \Ker F_{m+1}^0, g_2 \in \Ker F_{m+1}^1$. We will show that, for any $i,j= 1,2$,
$$\overline{f_i} \cup \overline{g_j} = \overline{f_i
\cup g_j} = 0.$$
The assumption $\mathcal{G}_n = \emptyset$ and Lemma  \ref{cero} imply that $f_1 = \sum_i \lambda_i
(\rho_i, \gamma_i) \in k(\mathcal{E}_n)$, so for any
$(\rho', \gamma') \in (AP_m // \mathcal{P})$ we have that
$$(\rho_i, \gamma_i)\cup (\rho', \gamma') = 0$$
because  $\rho_i\rho'\not \in AP_{n+m}$. Hence $f_1 \cup g_j=0$.   Similarly,  $ f_ i \cup g_1  = 0$.  Finally consider ${f_2} \cup {g_2}$ where
$$ f_2 =  \sum_{i} \mu_i (\rho_i, \gamma_i) \in \Ker F_{n+1}^1, \qquad
g_2 =  \sum_{j} \mu_j' (\rho_j', \gamma_j') \in \Ker F_{m+1}^1.$$
Recall that
$$ ^-(0,0)_n^- \sqcup (1,0)_n^-
\sqcup ^-(0,1)_n \sqcup (1,1)_n \sqcup (id +(-1)^{n}
\phi_{n})((1,0)_{n}^+) $$ is a basis of $\Ker F_{n+1}^1$ and observe that $f_2 \cup g_2  \in \Ker F^1_{n+m+1} \cap k( AP_{n+m} // \P_2)$.
If $(\rho_i, \gamma_i) \in { ^-(0,0)_n^- }$ or $(\rho_j',
\gamma_j')
 \in {^-(0,0)_m^-}$ then  $\gamma_i\gamma_j' \in I$ and hence  $(\rho_i\rho_j',\gamma_j\gamma_j') = 0$.  If
$(\rho_i, \gamma_i) \in { ^-(0,1)_n } $ and  $(\rho_j', \gamma_j')
 \in {(1,0)_m^-}$  we have that $$(\rho_i, \gamma_i) = (\hat{\rho}_i\alpha_1, \hat{\gamma}_i\alpha_1) \quad \mbox{and} \quad (\rho_j', \gamma_j') =  (\alpha_2\hat{\rho'}_j, \alpha_2\hat{\gamma'}_j)$$ with $\alpha_1,
\alpha_2 \in Q_1$.  Then
$(\rho_i\rho_j',\gamma_i\gamma'_j)  =
(\hat{\rho}_i\alpha_1\alpha_2\hat{\rho'}_j,\hat{\gamma}_i\alpha_1\alpha_2\hat{\gamma'}_j) = 0$ because $\alpha_1\alpha_2 \in I$ or $\hat{\rho}_i\alpha_1\alpha_2\hat{\rho'}_j \not \in AP_{n+m}$.
In all the remaining cases $(\rho_i\rho_j',\gamma_i\gamma_j')  \not \in {^-(0,0)_{n+m}^-}$, and the desired result follows from Proposition \ref{especial}. 
\end{proof}

\medskip

\begin{corollary}
Let $A = k Q/I$ be a quadratic string algebra, $\car k \not = 2$.  Then $\HH^n(A) \cup \HH^m(A) = 0$ for any $n, m >0$ odd natural numbers.
\end{corollary}

\begin{proof}
The assertion follows from the previous proof since for $n$ odd we have that $\Ker F^0_{n+1} = k \E_n$ and hence the hypothesis $\G_n = \emptyset$ is superfluous in this case. 
\end{proof}

Now we will describe the Lie bracket for gentle algebras.

\begin{lemma} \label{composicion} If $A = kQ/I$ is a gentle algebra,  $n, m > 1$ and $g \in {^-(0,0)_m^-}$.  Then $ f \circ  g = 0$ for any $f \in (AP_n//  \P)$.
\end{lemma}

\begin{proof}
Since  $f \circ g = \sum_{i=1}^n (-1)^{(i-1)(m-1)} f\circ_i g  $, it
suffices  to compute $ f\circ_i g$ for each $i$ with $1 \leq i \leq n$.  Let $g=(\rho, \gamma) = (\beta_1 \cdots \beta_m,  \gamma)$, with  $\gamma= \mu\gamma'= \gamma'' \nu$,  $\mu,  \nu \in Q_1$, $\beta_1 \neq \mu$ and $\beta_m \neq \nu$
and let $f= (\alpha_1 \ldots \alpha_n, \delta) \in ( AP_n//  \P)$. The non vanishing of $f \circ_i g$ would imply that 
 $$\alpha_1\ldots \alpha_{i-1}\beta_1 \ldots \beta_{m}\alpha_{i+1}
\ldots \alpha_n \in AP_{n+m-1} \mbox{ and } \alpha_i = \mu$$
$$\mbox{or}$$  $$\alpha_1\ldots \alpha_{i-1}\beta_1 \ldots
\beta_{m}\alpha_{i+1} \ldots \alpha_n \in AP_{n+m-1} \mbox{ and }
\alpha_i = \nu.$$
In this case
$$\alpha_{i-1}\beta_1 \in I \mbox{ and }
\alpha_{i-1}\mu \in I$$ $$ \mbox{or  }$$ $$ \beta_m\alpha_{i+1} \in
I \mbox{ and }  \nu\alpha_{i+1} \in I.$$
But $\beta_1 \neq \mu$, $\beta_m \neq \nu$ and by hypothesis $A$ is
gentle. This is a contradiction.  
\end{proof}

\begin{proposition}
If  $A = k Q/I$ is a  gentle algebra,  $n, m >1$, and $\G_{n-1} =   \emptyset = \G_{m-1}$
then  $[ \ \HH^n(A) , \HH^m(A)\  ] = 0$.
\end{proposition}

\begin{proof}
Let $\overline{f} = \overline{ f_1} + \overline{f_2} \in \HH^n(A)$, $\overline{g} =  \overline{g_1} + \overline{g_2} \in
\HH^m(A)$, with $f_1 \in \Ker F_{n+1}^0, f_2 \in \Ker F_{n+1}^1, g_1
\in \Ker F_{m+1}^0, g_2 \in \Ker F_{m+1}^1$. We will show that $ [  \overline{ f_i}   , 
 \overline{g_j}  ] =   0$ for any $i,j$.
Since $f_1 \in k(AP_n//Q_0)$ and $g_1 \in  k(AP_m//Q_0)$, it is clear that  ${f_i} \circ  {g_1} = 0 =  {g_i}\circ  {f _1}$ for any $i$. The statement is clear if $\overline f_2 =0 = \overline g_2$. 
If $\overline f_2 \not = 0$,  by Corollary \ref{gentil final} we may assume that $f_2$ belongs to $k ({^-(0,0)_n^-})$.  In this case
Lemma \ref{composicion} implies that  $ {g_i}\circ
{f _2}=0$  for any $i$.  The case $\overline g_2 \not =0$ is analogous, and hence we are done. 
\end{proof}

\begin{corollary}
Let $A = k Q/I$ be a gentle algebra, $\car k \not = 2$. Then $[
\HH^{n}(A), \HH^{m}(A) ] = 0$ for any $n, m  \geq 1$ even natural
numbers.
\end{corollary}

\begin{proof}
From the previous proof and Proposition \ref{gentle cero} we deduce that we only have to consider the case $f_2 \in k (1,1)_{n} \cap k (AP_{n} //Q_1)$ and $g_2 \in k (1,1)_{m} \cap k (AP_{m} //Q_1)$. Since $A$ is gentle we have that $(1,1)_{2s} \cap  (AP_{2s} //Q_1) = (1,1)_{2s}^{\G}$, and the proof of Proposition \ref{cero} shows that $k(1,1)_{2s}^{\G} = \Im G_{2s}^2 \subset \Im F_{2s}$, so $\overline f_2 = 0$ and $\overline g_2 = 0$ in this case. 
\end{proof}

\subsection{Non-vanishing of the cup product and the Lie bracket}

In this section we find conditions on the presentation $(Q,I)$ of a quadratic string algebra in order to get non-trivial structures.

\begin{theorem} \label{cup}
Let $A = kQ/I$ be a quadratic string algebra
and $\G_n \neq \emptyset$ for some  $n > 0$. Then the cup product
defined in $\HH^{*}(A)$ is non-trivial.  More
precisely,
\begin{itemize}
\item [(i)] if $n$ is even and $\car k \not =2$,  $\HH^{s_1n}(A) \cup \HH^{s_2n} (A) \not = 0$;
\item [(ii)] if $n$ is odd and $\car k \not =2$, $\HH^{2s_1n}(A) \cup \HH^{2s_2n} (A) \not = 0$;
\item [(iii)] if $\car k  =2$,  $\HH^{s_1n}(A) \cup \HH^{s_2n} (A) \not = 0$
\end{itemize}
for any $s_1, s_2 \geq 1$.
\end{theorem}

\begin{proof}
By hypothesis there exists  $\omega = ( \alpha_1 \cdots \alpha_n,  e_{s(\alpha_1)}) \in \G_n$.  If this element has order $r$ and $\car k =2$ or, $\car k \not = 2$ and $n$ is even, we have seen in the proof of Lemma \ref{cero} that
$$N (\omega) = N( \alpha_1 \cdots \alpha_n, e_{s(\alpha_1)})  = \sum_{i=0}^{r-1} \tau^i( \alpha_1 \cdots \alpha_n, e_{s(\alpha_1)}) \in \Ker F_{n+1}^0.$$
Moreover, for any $s \geq 1$, the element $\omega^s = ((\alpha_1
\cdots \alpha_n)^s, e_{s(\alpha_1)})$ belongs to $\G_{sn}$, has order $r$ and
$$N(\omega^s) = N((\alpha_1 \cdots \alpha_n)^s, e_{s(\alpha_1)}) = \sum_{i=0}^{r-1}\tau^i((\alpha_1 \cdots \alpha_n)^s, e_{s(\alpha_1)}) \in \Ker F^0_{sn+1}.$$
The element $N(\omega^s)$ is non-zero in $\HH^{sn}(A)$
because   $\Im F_{sn} \subseteq k(AP_{sn} // \P_1)$. For any
$s_1, s_2 \geq 1$ and $0 \leq i, j <r$, we have
$$\tau^i (\omega^{s_1}) \cup \tau^j  (\omega^{s_2}) = \delta_{ij} \tau^i (\omega^{s_1+s_2})$$
where $ \delta_{ij}$ is the Kronecker's delta. 
Then
$$N(\omega^{s_1}) \cup N(\omega^{s_2}) =  N(\omega^{s_1 + s_2})$$
and hence
$$\HH^{s_1n}(A) \cup \HH^{s_2n}(A) \neq (0).$$
If  $n$ is odd and $\car k \not =2$, we consider the elements $\omega^{2s}$ and we get that
$$\HH^{2s_1n}(A) \cup \HH^{2s_2n}(A) \neq (0)$$
for any $s_1, s_2 \geq 1$. 
\end{proof}

\begin{theorem} \label{lie}
Let $\car k =0$ and let $A = kQ/I$ be  a  quadratic string algebra  such that $\G_n
\neq \emptyset$ for some  $n > 0$. Then the Lie bracket defined
in $\HH^{*}(A)$ is non-trivial. More
precisely,
\begin{itemize}
\item [(i)] if $n$ is even,  $[\HH^{s_1n+1}(A) , \HH^{s_2n+1} (A)] \not = 0$;
\item [(ii)] if $n$ is odd,  $[\HH^{2s_1n+1}(A) , \HH^{2s_2n+1} (A)] \not = 0$
\end{itemize}
for any $s_1, s_2 \geq 1$, $s_1 \not = s_2$.
\end{theorem}

\begin{proof}
By hypothesis there exists  $\omega = ( \alpha_1 \cdots \alpha_n, e_{s(\alpha_1)})
\in \G_n$ and suppose  that  it has order $r$. From  Lemma
\ref{nucleo} we have  that
$$\psi (\omega)  =  (\alpha_1 \ldots \alpha_n\alpha_1, \alpha_1) \in \Ker
F^1_{n+2} $$
because $(\alpha_1 \ldots \alpha_n\alpha_1, \alpha_1)\in
(1,1)_{n+1}.$ 
Moreover, for any $s \geq 1$, the element 
$$\omega^s = ((\alpha_1
\cdots \alpha_n)^s, e_{s(\alpha_1)}) \in \G_{sn}$$ also has order $r$ and
$$\psi (\omega^s) = ((\alpha_1
\cdots \alpha_n)^s\alpha_1, \alpha_1) \in \Ker F^1_{sn+2} .$$
If $n$ is even, the element $\psi (\omega^s) $ is non zero in $\HH^{sn+1}(A)$
because we have seen in the proof of Lemma \ref{cero} that
$\psi (\omega^s)  \in \Im F_{sn+1}$ if and only if $w^s \in
\Im (1-\tau)$. But $\Im(1-\tau) = \Ker N$ and $w^s \not \in \Ker N$.
For any $s_1, s_2 \geq 1$, we have  that 
$$ \psi (\omega^{s_1}) \circ_i \psi (\omega^{s_2}) = \left\{ \begin{array}{ll}
\psi (\omega^{s_1+ s_2}), & \hbox{if $i = lr + 1 $, for $l = 0, \cdots , s_1\frac{n}{r}$;} \\
0, & \hbox{otherwise.} \end{array} \right.$$
Then 
 \begin{align*}  
\psi (\omega^{s_1}) \circ  \psi (\omega^{s_2}) &  =  
\sum_{i=1}^{s_1n+1} (-1)^{(i-1)(s_2n)}  \psi (\omega^{s_1}) \circ_i  \psi (\omega^{s_2})
= \sum_{l=0}^{s_1\frac{n}{r}} (-1)^{(lr)(s_2n)}  \psi (\omega^{s_1+ s_2}) \\
& =  (s_1\frac{n}{r} + 1) \psi (\omega^{s_1 + s_2 })\end{align*}
and hence
\begin{align*}  [  \psi (\omega^{s_1}) ,  \psi (\omega^{s_2}) ]  & =    \psi (\omega^{s_1}) \circ  \psi (\omega^{s_2}) - (-1)^{(s_1n)(s_2n)}   \psi (\omega^{s_2}) \circ  \psi (\omega^{s_1})  \\
& =    (s_1\frac{n}{r} + 1)  \psi (\omega^{s_1 + s_2})-  (s_2\frac{n}{r} +
1)  \psi (\omega^{s_1+ s_2 }) \\
& =   \frac{n}{r}(s_1 - s_2)  \psi (\omega^{s_1 + s_2}).
\end{align*}
So, $$[\  \HH^{ns_1 + 1 }(A) ,  \HH^{ns_2 + 1}(A) \ ] \neq 0 \mbox{
if } s_1 \neq s_2. $$ 
If  $n$ is odd, we consider the element
$$ \psi (\omega^{2s}) = ( (\alpha_1 \ldots \alpha_n)^{2s}\alpha_1, \alpha_1) \in \Ker
F_{2sn+2}$$
and we get that
$$[ \ \HH^{2ns_1 +1 }(A) , \HH^{2ns_2 + 1 }(A) \ ] \neq 0,
 \mbox { if } s_1 \neq s_2.$$ 
\end{proof}

\begin{remark}
The previous theorem also holds in any characteristic if we add some hypothesis on the number $\frac{n}{r}$.
\end{remark}

\begin{example}   For the algebra in  Example \ref{ejemplo2} we have that $\G_n  \neq  \emptyset$  hence our Theorems \ref{lie} and \ref{cup} say that its Gerstenhaber algebra structure is  non-trivial. More precisely, if  we consider the following sets of generators:  
\begin{itemize}
\item $\{ u_1 = (\beta_1, \beta_1), v_1= (\alpha_1, \alpha_1), w_1= (\alpha_7, \alpha_7) \}$ for $\HH^{1}(A),$
\item $\{ u_7 = ((\beta_1\beta_2)^{3}\beta_1, \beta_1), v_7= (\alpha_1\dots\alpha_7, e_1) \}$ for $\HH^{7}(A),$
\item $ \{( u_{2n+1} = ((\beta_1\beta_2)^{n}\beta_1, \beta_1) \}$ for $\HH^{2n+1}(A)$ when $n>0$, $n \neq 3$, and
\item  $\{ u_{2n}= N(  ((\beta_1\beta_2)^{n}, e_6)  ) \}$ for $\HH^{2n}(A)$ when $n \geq 1$, 

\end{itemize}
then 
\begin{align*}
 u_{2i} \cup  u_{2j} & =  u_{2(i+j)} \\
 u_{2i} \cup  u_{2j+1} & =    u_{2j+1} \cup  u_{2i} \ = \ u_{2(i+j)+1}  \\
 {[}  u_{2i+1} , u_{2j+1}  {]}   & =     (i-j) \ u_{2(i+j)+1} \\
 {[}  u_{2i} , u_{2j+1}  { ]} & =   - [ \  u_{2j+1} , u_{2i}  \  ] \ = \ i \ u_{2(i+j)} \\
{[}  v_{7} , v_{1}  {]}   & =    - {[}  v_{1} , v_{7}  {]} = v_{7}  \\
 {[}  v_{7} , w_{1}  {]}   & =    - {[}  w_{1} , v_{7}  {]} = v_{7}  \
\end{align*}
and all the other cases are zero.  
\end{example}

 \begin{example}   Let $\car k \not =2$  and let $A=kQ/I$ be the radical square zero algebra whose quiver $Q$ consists of  an oriented cycle of length $m >1$
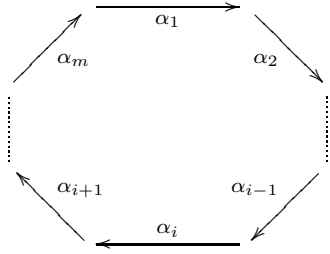
\[ \xymatrix{
&  & \ar[rr]_{\alpha_1}  & &    \ar[rd]_{\alpha_2} & & \\  & \ar[ru]_{\alpha_{m}}  &  & & & \ar@{.}[d] &  \\ &  \ar@{.}[u]  & & & & \ar[ld]_{\alpha_{i-1}}    &  \\   &  & \ar[lu]_{\alpha_{i+1}} &  & \ar[ll]_{\alpha_i} & & 
} \] 
A straightforward computation shows that
$$   \HH^{qm}(A) =k, \quad  \HH^{qm+1}(A) =k$$
when $qm$ is even and it is zero in all the other cases.  More precisely,  
$$u_{qm}=N(( \alpha_1, \dots, \alpha_m)^q, e_{s(\alpha_1)}) \quad \mbox{and} \quad  u_{qm+1}=(( \alpha_1\dots, \alpha_m)^q\alpha_1, \alpha_1)$$
are   generators of $\HH^{qm} (A)$ and 
$\HH^{qm+1} (A)$ respectively. 
From Theorems \ref{cup}  and \ref{lie} we know that its Gerstenhaber algebra structure is non-trivial and the bracket on the generators is given by 
\begin{align*}  [ u_{q_1m+1}, u_{q_{2m+1}}]   &  =    (q_1-q_2) \ u_{(q_1+q_2)m+1} \\
{[}  u_{ q_1m }, u_{ q_{ 2m+1 } } {]} &   =     -{[}   u_{ q_{ 2m+1 } }, u_{ q_1m }  {]} = q_1 \ u_{ (q_1+q_2)m  }
\end{align*}
If we only consider odd degrees, that is, $$ \HH^{odd}(A) = \bigoplus_{\substack{q \geq0, \\  \ qm \mbox{ \small even }} } \HH^{qm+ 1}(A),$$   then the Gerstenhaber bracket endows $ \HH^{odd}(A)$ with a Lie algebra structure such that it is isomorphic to the infinite dimensional Witt algebra.  In this way we recover the results in \cite[Proposition 4.3.4]{SF1}. 

\end{example}

\end{document}